\documentclass[12pt]{amsart}

\usepackage[dvipdfmx]{graphicx}
\usepackage{amssymb,url}
\usepackage{url}

\newtheorem{theorem}{Theorem}[section]
\newtheorem{proposition}[theorem]{Proposition}
\newtheorem{lemma}[theorem]{Lemma}
\newtheorem{corollary}[theorem]{Corollary}

\theoremstyle{definition}

\newtheorem{example}[theorem]{Example}
\newtheorem{problem}[theorem]{Problem}
\newtheorem{question}[theorem]{Question}

\newtheorem{remark}[theorem]{Remark}

\newtheorem*{proposition3}{Proposition \ref{3.3}}

\newcommand{\NN}{ \ensuremath{\mathbb{N}}}
\newcommand{\ZZ}{ \ensuremath{\mathbb{Z}}}

\newcommand{\RR}{ \ensuremath{\mathbb{R}}}

\newcommand{\reg}{\mathop{\mathrm{reg}}}

\newcommand{\Tor}{\ensuremath{\mathrm{Tor}}\hspace{1pt}}

\newcommand{\Image}{\mathop{\mathrm{Image}}\hspace{1pt}}

\newcommand{\aaa}{\mathbf{a}}
\newcommand{\bb}{\mathbf{b}}

\newcommand{\st}{\mathrm{st}}

\newcommand{\sym}{\ensuremath{\mathfrak{S}}}

\def\cocoa{{\hbox{\rm C\kern-.13em o\kern-.07em C\kern-.13em o\kern-.15em A}}}

\newcommand{\kk}{\Bbbk}
\newcommand{\II}{\mathcal{I}}
\newcommand{\LL}{\mathcal{J}}
\newcommand{\pd}{\mathop{\mathrm{pd}}}
\newcommand{\ua}{{\underline {\mathbf{a}}}}
\newcommand{\supp}{\mathrm{supp}}

\begin{document}

\title[symmetric monomial ideals]
{Betti tables of monomial ideals fixed by permutations of the variables}

\author{Satoshi Murai}
\address{
Satoshi Murai,
Department of Mathematics
Faculty of Education
Waseda University,
1-6-1 Nishi-Waseda, Shinjuku, Tokyo 169-8050, Japan}
\email{s-murai@waseda.jp}



\begin{abstract}
Let $S_n$ be a polynomial ring with $n$ variables over a field and $\{I_n\}_{n \geq 1}$ a chain of ideals such that each $I_n$ is a monomial ideal of $S_n$ fixed by permutations of the variables.
In this paper, we present a way to determine all nonzero positions of Betti tables of $I_n$ for all large intergers $n$ from the $\mathbb Z^m$-graded Betti table of $I_m$ for some integer $m$.
Our main result shows that the projective dimension and the regularity of $I_n$ eventually become linear functions on $n$,
confirming a special case of conjectures posed by Le, Nagel, Nguyen and R\"omer.
\end{abstract}

\maketitle

\section{Introduction}

Recently,
ideals fixed by an action of the infinite symmetric group in a polynomial ring with infinitely many variables attract the interest of researchers in various areas of mathematics.
For example,
a finite generation property of such ideals up to symmetry 
has been interested in algebraic statistics
(see \cite{AH,HS} and a survey \cite{Dr}).
From representation theory point of view, a study of such ideals can be considered as a special instance of twisted commutative algebra \cite{SS,SS2} and FI-modules \cite{EFS}.
In this paper,
motivated by commutative algebra questions posed by Le, Nagel, Nguyen and R\"omer \cite{LNNR1,LNNR2},
we study Betti tables of monomial ideals fixed by permutations of the variables.

Let $S_\infty=\kk[x_j: j \geq 1]$ be a polynomial ring over a field $\kk$ with infinitely many variables $x_1,x_2,\dots$ and let $S_n=\kk[x_1,\dots,x_n]$.
Consider the action of the infinite symmetric group $\sym_\infty$ to $S_\infty$ defined by $\sigma(x_i)=x_{\sigma(i)}$
for any $\sigma \in \sym_\infty$,
and consider an ideal $\II \subset S_\infty$ which is fixed by the action of $\sym_\infty$, that is, satisfies $\sigma(\II)=\II$ for any $\sigma \in \sym_\infty$.
Given such an ideal $\II$, by setting $I_n=\II \cap S_n$
we obtain a chain of ideals
\begin{align}
\label{chain}
I_1 \subset I_2 \subset I_3 \subset \cdots
\end{align}
such that each $I_n$ is fixed by the action of the $n$th symmetric group $\sym_n$.
About such a chain,
a natural interesting algebraic problem is to understand asymptotic behavior of $I_n$.
Indeed, Le, Nagel, Nguyen and R\"omer \cite{LNNR1,LNNR2} recently studied
asymptotic behavior of projective dimension and regularity of $I_n$.
They give certain linear bounds for these invariants and conjectured that they become linear functions on $n$ for $n \gg 0$.
The above setting of considering $\sym_\infty$-invariant ideals in $S_\infty$ is actually a special case of their problems
since they actually discussed chains $\{I_n\}$
of ideals such that each $I_n$ is an ideal of $\kk[x_{i,j}:1\leq i \leq c, 1 \leq j \leq n]$. But even for chains arising from ideals in $S_\infty$ asymptotic behavior of these homological invariants are still not understood very well.
The purpose of this paper is to explain that,
in the special case when $\II$ is a monomial ideal in $S_\infty$,
the situation becomes quite simple and one can describe asymptotic behavior of not only projective dimension and regularity but also the shape of the Betti table.

To explain our main result,
let us first introduce a few notation and give one simple example.
A {\bf partition of length $k$} is a sequence $\lambda=(\lambda_1,\dots,\lambda_k)$ of positive integers satisfying $\lambda_1 \geq \cdots \geq  \lambda_k$.
For a vector $\aaa=(a_1,\dots,a_k) \in \ZZ^k_{\geq 0}$,
we write $x^\aaa=x_1^{a_1}x_2^{a_2} \cdots x_k^{a_k}$.
We say that an ideal $\II \subset S_\infty$ is {\bf symmetric} if it is fixed by the action of $\sym_\infty$.
We can consider that a symmetric monomial ideal $\II \subset S_\infty$ is generated by a finite number of partitions in the sense that there always exist partitions $\lambda(1),\dots,\lambda(t)$ such that
$$\II=\big( \sigma(x^{\lambda(k)}):k=1,2,\dots t,\ \sigma \in \sym_\infty \big)$$
(see Section 2).
The {$(i,j)$th \bf graded Betti number} of a homogeneous ideal $I \subset S_n$
is the number $\beta_{i,j}(I)=\dim_\kk \Tor_i(I,\kk)_j$.
To write down graded Betti numbers,
we use the {\bf Betti table} of $I$, which is the table whose $(i,j)$th entry is the number $\beta_{i,i+j}(I)$.

Now let us give a simple example.
Let $\LL \subset S_\infty$ be the symmetric monomial ideal generated by partitions $(5,1)$ and $(2,2)$,
and let $J_n=\LL\cap S_n$ for $n \geq 1$.
Thus $J_n$ is the monomial ideal of $S_n$ generated by the $\sym_n$-orbits of $x_1^5x_2$ and $x_1^2x_2^2$.
In Figures \ref{fig1} and \ref{fff2} below,
we list Betti tables of $J_n$ when $n=2,3,4,5,6$ and $9$ computed by the computer algebra system Macaulay2 \cite{M2}.

\begin{figure}[h]
\includegraphics[scale=0.35]{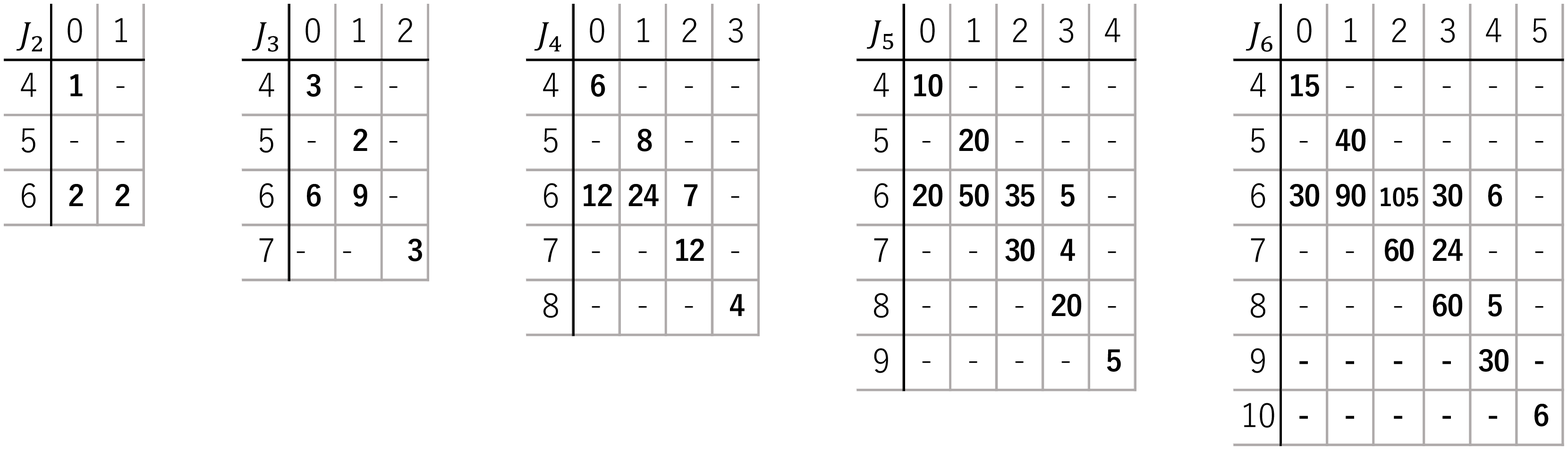}
\caption{Betti tables of $J_n$ for $n=2,3,4,5,6$.}
\label{fig1}
\end{figure}

\begin{figure}[h]
\includegraphics[scale=0.3]{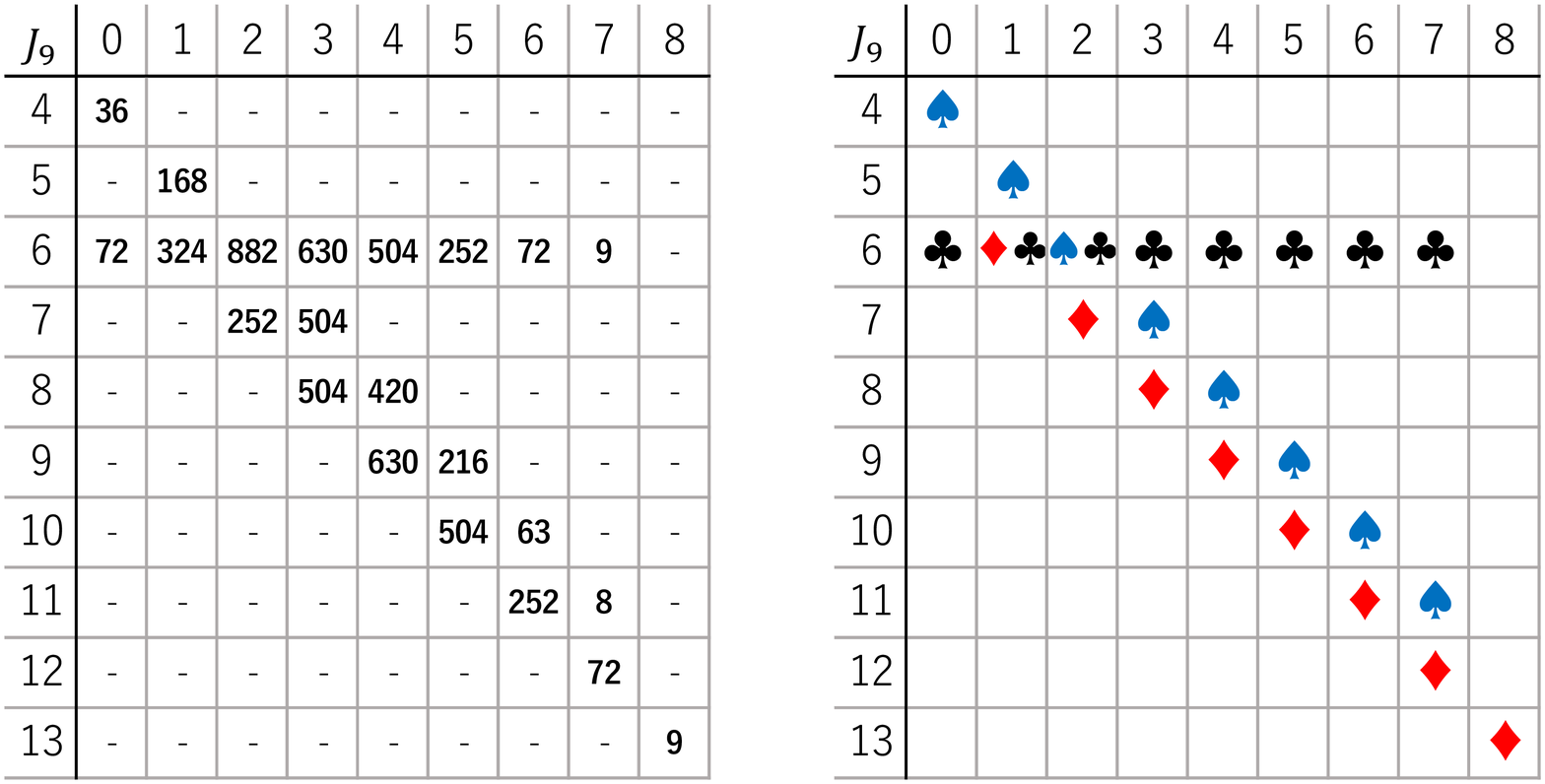}
\caption{Betti table of $J_9$. Each suit represents a line segment.}
\label{fff2}
\end{figure}

\noindent
Looking these tables, one can find that the shape of the nonzero positions in the Betti table of $J_n$ looks like a union of ``line segment of length $n-2$".
Here, for ``line segment of length $\ell$", we mean a set of the form 
\[
\mathcal L\big((i,j),c,\ell\big)=\{(i+k,j+ck) \in \ZZ^2:k=0,1\dots,\ell\}\]
for some integers $i,j,c,\ell \in \ZZ_{\geq 0}$
(one may think that, in the symbol $\mathcal L\big((i,j),c,\ell\big)$,
the first entry $(i,j)$ represents the starting position, $c$ represents the slope, and $\ell$ represents the length). 
Indeed, for our example of the ideal $\LL$,
it is not hard to show (see \S 2.4 later)
\begin{align}
\label{0-1}
&\{(i,j)\in \ZZ^2: \beta_{i,i+j}(J_n) \ne 0\}
\\
\nonumber
&= \mathcal L\big((0,4),1,n-2\big)
\cup
\mathcal L\big((0,6),0,n-2\big)
\cup 
\mathcal L\big((1,6),1,n-2\big).
\end{align}
See the right table in Figure \ref{fff2}.
The main result of this paper is the following result which shows
that this phenomenon always happens if we ignore some positions in low homological degrees.

\begin{theorem}
\label{1.2}
Let $\II \subset S_\infty$ be a symmetric monomial ideal generated by partitions of length $\leq m$ and let $I_n=\II \cap S_n$ for $n \geq 1$.
Then there is a finite set $D \subset \{0,1,\dots,m-1\} \times \ZZ_{\geq 0}^2$ such that for any integer $n \geq m$ we have
\begin{align*}
\big\{(i,j): \beta_{i,i+j}(I_n) \ne 0\big\}
=
\left(\bigcup_{(i,j,c) \in D} \!\! \mathcal L\big((i,j),c,n-m\big)\!\!\right)
\!\cup \big\{ (i,j): \beta_{i,i+j}(I_{m-1}) \ne 0\big\}.
\end{align*}
\end{theorem}

To prove Theorem \ref{1.2},
we actually prove a $\ZZ^n$-graded version of the theorem
which enable us to determine all nonzero positions of the ($\ZZ^n$-graded) Betti table of $I_n$ from the $\ZZ^m$-graded Betti table of $I_m$.
We do not explain this result here since the statement is not very simple.
See Theorem \ref{3.4} later.

Theorem \ref{1.2} immediately proves the following corollary about projective dimension and regularity,
which confirms a special case of conjectures given by Le, Nagel, Nguyen and R\"omer in \cite[Conjecture 1.1]{LNNR1} and \cite[Conjecture 1.2]{LNNR2}.


\begin{corollary}
\label{1.3}
Let $\II \subset S_\infty$ be a nonzero proper symmetric monomial ideal and let $I_n=\II \cap S_n$ for $n \geq 1$. There are integers $D,W,C$ such that
\[\pd(I_n)=n-D \mbox{ and }\reg(I_n)=Wn+C \mbox{ for }n \gg 0.\]
\end{corollary}

We later show that the integer $W$ can be determined combinatorially and $\pd(I_n)$ stabilize in a quite early stage.
See Corollary \ref{3.9} and Proposition \ref{3.11}.

There is one more interesting consequence of Theorem \ref{1.2}.
The result of Nagel and R\"omer in \cite[Theorem 7.7]{NR} tells
that for any symmetric ideal $\II \subset S_\infty$ and an integer $p\geq 0$,
the set
$$\{j: \beta_{p,j}(\II \cap S_n) \ne 0\}$$
stabilize for $n \gg 0$.
Theorem \ref{1.2} shows that the following stronger property holds when $\II$ is a monomial ideal.

\begin{corollary}
\label{1.4}
Let $\II \subset S_\infty$ be a symmetric monomial ideal and let $I_n=\II \cap S_n$ for $n \geq 1$. There is an integer $M$ such that 
$$|\{j: \beta_{p,j}(\II \cap S_n) \ne 0\}|=M \ \ \mbox{ for $p,n \gg 0$}.$$
\end{corollary}

Nagel and R\"omer actually proved that, for a fixed $p$, there are finite number of syzygies that create all $p$th syzygies of $I_n$ when $n \gg 0$ (see also \cite{Sn} for a related result).
We think that the above corollary suggests a possibility that there might be a way to create not only all $p$th syzygies for a fixed $p$ but also all $p$th syzygies for arbitrary $p$ from a finite list of syzygies.

The results on projective dimension and regularity in this paper (Corollaries \ref{1.3}, \ref{3.9}, \ref{3.10}, \ref{constant} and Proposition \ref{3.11}) are also proved by Claudiu Raicu \cite{Ra} independently by a different method. He actually find a formula of $\pd(I)$ and $\reg(I)$ for any symmetric monomial ideal $I$ in $S_n$ and his formula determines $C,D,W$ in Corollary \ref{1.3}.

This paper is organized as follows:
In Section 2, we discuss basic properties of symmetric monomial ideals and their $\ZZ^n$-graded Betti numbers.
In Section 3, we prove our main result 
postponing the proof of our key combinatorial proposition,
and prove some refinements of Corollary \ref{1.3}.
In Section 4, we prove our key combinatorial proposition, and in Section 5 we present some open questions.

\section{Preliminary}

In this section, we discuss basic properties of symmetric monomial ideals and multigraded Betti numbers.

\subsection{Symmetric monomial ideals}
By the result of Aschenbrenner and Hiller \cite{AH}, for any symmetric ideal $\II \subset S_\infty$ there are polynomials $f_1,\dots,f_t \in S_\infty$ such that
$$\II=\big( \sigma(f_k):1 \leq k \leq t,\ \sigma \in S_\infty \big).$$
In that case, we say that $\II$ is generated by $f_1,\dots,f_t$.
A monomial ideal of $S_\infty$ is an ideal generated by monomials.
Let $\Lambda$ be the set of all partitions.
For any monomial $x^\aaa$,
there is a unique partition $\lambda$ such that $x^\aaa=\sigma(x^\lambda)$ for some $\sigma \in \sym_\infty$.
Thus, for any symmetric monomial ideal $\II \subset S_\infty$, there are partitions $\lambda(1),\dots,\lambda(t)$ such that $x^{\lambda(1)},\dots,x^{\lambda(t)}$ generates $\II$. We identify each $\lambda(k)$ with the monomial $x^{\lambda(k)}$ and say that $\lambda(1),\dots,\lambda(t)$ generate $\II$.
It is easy to see that for any symmetric monomial ideal $\II\subset S_\infty$, there is the unique minimal subset of $\Lambda$ that generates $\II$.
This set will be denoted by $\Lambda(\II)$.

We also say that a monomial ideal $I \subset S_n$ is symmetric if $I$ is fixed by the action of $\sym_n$.
In the same way as for ideals in $S_\infty$,
for any symmetric monomial ideal $I \subset S_n$,
there are partitions $\lambda(1),\dots,\lambda(t)$ of length $\leq n$ such that
$I=(\sigma(x^{(\lambda(k)}):1 \leq k \leq t, \sigma \in \sym_n)$.
We denote by $\Lambda(I)$ the unique minimal subset of $\Lambda$ that generates $I$.
Note that if $\II$ is a symmetric monomial ideal of $S_\infty$, then the ideal $I_n=\II\cap S_n$ is a symmetric monomial ideal of $S_n$ with
$$\Lambda(I_n)=\{\lambda \in \Lambda(\II): \lambda \mbox{ has length } \leq n\}.$$
For example, if $\II$ is generated by $(3,3)$ and $(2,2,2)$, then
\begin{align*}
I_1=\{0\},\ I_2=(x_1^3x_2^3),\
I_3=(x_1^3x_2^3,x_1^3x_3^3,x_2^3x_3^3,
x_1^2x_2^2x_3^2),
\dots.
\end{align*}
We often use the following property of symmetric monomial ideals.

\begin{lemma}
\label{tech}
Let $I \subset S_n$ be a symmetric monomial ideal generated by partitions of length $\leq m$.
Suppose $n >m$.
For any vector $\aaa=(a_1,\dots,a_n)\in \ZZ_{\geq 0}^n$ with $a_1,\dots,a_{m} \geq a_n$,
if $x^\aaa \in I$ then $x_1^{a_1} \cdots x_{n-1}^{a_{n-1}} \in I$.
\end{lemma}

\begin{proof}
Since $I$ is generated by partitions of length $\leq m$, if $x^\aaa \in I$,
then there is a monomial $u=x_{i_1}^{b_1} \cdots x_{i_m}^{b_m} \in I$ with $i_1 < \cdots <i_m$ that divides $x^\aaa$.
If $i_m<n$, this monomial $u$ clearly divides $x_1^{a_1} \cdots x_{n-1}^{a_{n-1}}$.
Suppose $i_m =n$.
Then there is a variable $x_k$ with $k \leq m$ that does not appear in $x_{i_1}^{b_1} \cdots x_{i_m}^{b_m}$ and the monomial $u'=x_{i_1}^{b_1} \cdots x_{i_{m-1}}^{b_{m-1}} x_{k}^{b_m} \in I$ by the symmetry of $I$.
This monomial $u' \in I$ divides  $x_1^{a_1} \cdots x_{n-1}^{a_{n-1}}$ since $a_k \geq a_n \geq b_m$.
\end{proof}

\subsection{Betti numbers via simplicial complexes}

When studying graded Betti numbers of monomial ideals,
it is standard in commutative algebra to consider their multidegrees, that is, their $\ZZ^n$-gradings. An advantage of considering multidegrees is the fact that $\ZZ^n$-graded Betti numbers of monomial ideals can be computed from certain simplicial complexes.
We quickly recall this fact.

Consider the $\ZZ^n$-grading of $S_n$ such that the degree of $x_i$ is the $i$th standard vector of $\ZZ^n$. For a finitely generated $\ZZ^n$-graded $S_n$-module $M$ and $\aaa=(a_1,\dots,a_n) \in \ZZ^n$,
we write $M_\aaa$ for its graded component of degree $\aaa$ and call the numbers $\beta_{i,\aaa}(M)=\dim_\kk \Tor_i(M,\kk)_\aaa$ the {\bf $\ZZ^n$-graded Betti numbers} of $M$.

A simplicial complex $\Delta$ on $[n]=\{1,2,\dots,n\}$ is a collection of subsets of $[n]$ satisfying that $F \in \Delta$ and $G \subset F$ imply $G \in \Delta$.
(We do not assume that every singleton of $[n]$ is contained in $\Delta$.)
Elements in $\Delta$ are called {\bf faces} of $\Delta$
and faces having cardinality $1$ are called {\bf vertices} of $\Delta$.
We denote by $\widetilde H_i(\Delta)$ the $i$th reduced homology group of $\Delta$ over a field $\kk$.

Let $\aaa=(a_1,\dots,a_n)\in \ZZ^n_{\geq 0}$
and $\ua=(a_1-1,\dots,a_n-1)$.
For a monomial ideal $I \subset S_n$,
we define the simplicial complex
$$
\Delta_\aaa^I =
\left\{F \subset [n]: \frac {x^\aaa} {x^F} \in I\right\}= \{F \subset [n] : x^\ua \cdot x^{[n] \setminus F} \in I\}
$$
where $x^F=\prod_{i \in F} x_i$.
Note that we consider that $\frac {x^\aaa} {x^F}$ and $x^{\ua} \cdot x^{[n] \setminus F}$ are not in $I$ if they have negative exponents.
The following fact is known.

\begin{lemma}
\label{2.1}
For any monomial ideal $I \subset S_n$ and $\aaa \in \ZZ_{\geq 0}^n$, we have
$$\Tor_i(I,\kk)_\aaa \cong \widetilde H_{i-1}(\Delta_\aaa^I) \ \ \mbox{ for all }i \geq 0.$$
\end{lemma}

Indeed, the degree $\aaa$ homogeneous component of the Koszul complex of $I$ w.r.t.\ the variables $x_1,\dots,x_n$ can be identified with the simplicial chain complex of $\Delta_\aaa^I$.
See \cite[Theorem 1.34]{MS}.

Here we also recall a few basic facts on $\ZZ^n$-graded Betti numbers and homologies of simplicial complexes.
Let $I \subset S_n$ be a monomial ideal and $G(I)$ the minimal set of monomial generators of $I$.
The set of all lcms of monomials in $G(I)$ is called the {\bf lcm lattice of $I$} and will be denoted by $\mathrm{Lcm}(I)$.
For a vector $\aaa=(a_1,\dots,a_n) \in \ZZ_{\geq 0}^n$, we write
$$\supp(\aaa)=\{i: a_i>0\}$$
and $\supp(x^\aaa)=\supp(\aaa)$.

\begin{lemma}
\label{2.2}
Let $I \subset S_n$ be a monomial ideal and $\aaa \in \ZZ_{\geq 0}^n$.
\begin{enumerate}
\item[(i)] If $\aaa \not \in \mathrm{Lcm}(I)$ then $\beta_{i,\aaa}(I)=0$.
\item[(ii)] If there is an element $m \in G(I)$ such that $m$ divides $x^\aaa$ and
$\supp(x^\aaa)=\supp(x^\aaa/m)$ then $\beta_{i,\aaa}(I)=0$.
\end{enumerate}
\end{lemma}

The first statement is an easy consequence of Taylor resolutions (see \cite[\S 7.1]{HH}),
and the second statement follows since $\Delta^I_\aaa$ becomes the $(n-1)$-simplex $\{F : F \subset [n]\}$ under the assumption.
See the proof of \cite[Theorem 3.2]{BPS}.

We also need the following fact.

\begin{lemma}
\label{2.3}
Let $I \subset S_n$ be a monomial ideal and $\aaa=(a_1,\dots,a_n) \in \ZZ_{\geq 0}^n$.
\begin{enumerate}
\item[(i)] If $a_n=0$ then $\beta_{i,\aaa}(I)=\beta_{i,(a_1,\dots,a_{n-1})}(I\cap S_{n-1})$.
\item[(ii)]
If $\beta_{i,\aaa}(I) \ne 0$ then $i < |\supp(\aaa)|$.
\end{enumerate}
\end{lemma}

\begin{proof}
The first statement follows from the fact that $\Delta^I_{(a_1,\dots,a_{n-1},0)}=\Delta^{I \cap S_{n-1}}_{(a_1,\dots,a_{n-1})}$.
We prove (ii).
Suppose $t =|\supp(\aaa)|$.
We may assume $\aaa=(a_1,\dots,a_t,0,\dots,0)$.
Then the first statement tells 
$\beta_{i,(a_1,\dots,a_t)}(I\cap S_t)=\beta_{i,\aaa}(I) \ne 0$.
Since $I \cap S_t$ is an ideal in the polynomial ring with $t$ variables,
$i$ must be smaller than $t=|\supp(\aaa)|$.
\end{proof}

We say that a simplicial complex is {\bf acyclic} if all its homology groups are zero.
Let $\Delta \subset \Gamma$ be simplicial complexes.
Then there is a natural map of homologies
$\iota: \widetilde H_i(\Delta) \to \widetilde H_i(\Gamma)$
induced by the inclusion $\Delta \subset \Gamma$.
We often use the next fact.

\begin{lemma}
\label{3.1}
Let $\Delta \subset \Gamma \subset \Sigma$ be simplicial complexes.
If $\Gamma$ is acyclic, then the map 
$\iota: \widetilde H_i(\Delta) \to \widetilde H_i(\Sigma)$
induced by the inclusion $\Delta \subset \Sigma$ is zero for any $i$.
\end{lemma}

\begin{proof}
The statement immediately follows from the fact that $\iota$ equals to the composition
of the two maps 
$\widetilde H_i(\Delta) \to \widetilde H_i(\Gamma)$ and $\widetilde H_i(\Gamma) \to \widetilde H_i(\Sigma)$ induced by inclusions.
\end{proof}

\subsection{Multigraded Betti numbers of symmetric monomial ideals}

If $I \subset S_n$ is a symmetric monomial ideal, then $\Tor_i(I,\kk)$ admit an action of $\sym_n$ induced by the action on $I$.
By this action, $\sigma \in \sym_n$ sends each element $f \in \Tor_i(I,\kk)$ of degree $(a_1,\dots,a_n)$ to an element of degree $(a_{\sigma(1)},\dots,a_{\sigma(n)})$.
Thus, to study $\ZZ^n$-graded Betti numbers of $I$, it is enough to consider degrees $(a_1,\dots,a_n)$ with $a_1 \geq \cdots \geq a_n$.

Let $\Delta$ be a simplicial complex.
We say that $\Delta$ is a cone with apex $v$ if
$$\Delta=\Delta' \cup \{\{v\} \cup F: F \in \Delta'\}$$
where $\Delta'=\{F \in \Delta: v \not \in F\}$.
It is standard in combinatorial topology that if $\Delta$ is a cone, then $\Delta$ is acyclic.
The next statement is easy to prove, but gives a strong restriction to possible multidegrees for Betti numbers of symmetric monomial ideals.

\begin{proposition}
\label{2.4}
Let $I \subset S_n$ be a symmetric monomial ideal generated by partitions of length $\leq m$
and $\aaa=(a_1,\dots,a_t,0,\dots,0)\in \ZZ_{\geq 0}^n$ with $a_1 \geq \cdots \geq a_t \geq 1$.
If $m < t$ and $a_m > a_t$, then $\beta_{i,\aaa}(I)=0$ for all $i$.
\end{proposition}

\begin{proof}
We prove that the simplicial complex
$$\Delta_\aaa^I=\{ F \subset [n]: x^\ua \cdot  x^{[n] \setminus F} \in I\}
=\{F \subset [t]: x_1^{a_1-1}\cdots x_t^{a_t-1} x^{[t] \setminus F} \in I\}$$
is a cone with apex $t$.
Let $F \in \Delta^I_\aaa$ with $t \not \in F$.
What we must prove is that $F \cup \{t\} \in \Delta^I_\aaa$.
Since $F \in \Delta^I_\aaa$, we have 
\[
x_1^{a_1-1}\cdots x_t^{a_t-1} x^{[t] \setminus F}
=x_1^{a_1-1}\cdots x_{t-1}^{a_{t-1}-1} \cdot  x^{[t] \setminus (F \cup \{t\})} \cdot x_t^{a_t} \in I.
\]
Since $m<t$ and $a_m >a_t$,
Lemma \ref{tech} tells 
$(x_1^{a_1-1}\cdots x_{t-1}^{a_{t-1}-1})x^{[t]\setminus (F \cup \{t\})} \in I$
and we have $(x_1^{a_1-1}\cdots x_{t-1}^{a_{t-1}-1} x_t^{a_t-1})x^{[t]\setminus (F \cup \{t\})} \in I$.
This tells $F \cup \{t\} \in \Delta_\aaa^I$ as desired.
\end{proof}

Another expression of Proposition \ref{2.4} is that, with the same notation as in the proposition, if $\beta_{i,\aaa}(I_n) \ne 0$  for some $i$ and $n \geq m$, then $\aaa$ must be of the form
$$\aaa=(a_1,\dots,a_{m-1},a_m,a_m,\dots,a_m,0,\dots,0).$$

\subsection{Warm up: Ideals generated by partitions of length $2$}

In this subsection, to get some feelings about (multigraded) Betti tables of symmetric monomial ideals,
we discuss a very special case when the ideal is generated by partitions of length $2$.
This subsection can be considered as a special case of a more general result proved later and can be skipped if the reader is just interested in the proof of the main result.

For integers $a$ and $\ell$, we write $(a^\ell)=(a,\dots,a) \in \ZZ^\ell$.
We also define $(a_1^{\ell_1},\dots,a_t^{\ell_t})\in \ZZ^{\ell_1+ \cdots + \ell_t}$ similarly.
Let $\II \subset S_\infty$ be a symmetric monomial ideal generated by partitions of length $2$.
Then $\Lambda(\II)$ must be a set of the form
$$\Lambda(I)=\{(p_1,q_1),(p_2,q_2),\dots,(p_t,q_t)\}$$
with
$$p_1>p_2> \cdots >p_t \geq q_t > \cdots > q_2 > q_1 \geq 1.$$
In this situation, all possible $\ZZ^n$-graded Betti numbers are determined in the following way.

\begin{proposition}
\label{2.6}
Let $\II$ and $(p_1,q_1),\dots,(p_t,q_t)$ be as above.
Fix $n \geq 2$ and let $I_n=\II \cap S_n$. Then
$\beta_{i,\aaa}(I) \in \{0,1\}$ for any $\aaa \in \ZZ_{\geq 0}^n$ and
\begin{align*}
&\big\{ \big(i,(a_1,\dots,a_n)\big): \beta_{i,(a_1,\dots,a_n)}(I_n) \ne 0, a_1 \geq \cdots \geq a_n\big\}\\
&=\left(\bigcup_{k=1}^t \big\{\big(i,(p_k,q_k^{i+1},0^{n-i-2})\big):i=0,1,\dots,n-2\big\}\right)\\
&\hspace{16pt} \cup \left(\bigcup_{k=1}^{t-1} \big\{\big(i+1,(p_k,q_{k+1}^{i+1},0^{n-i-2})\big):i=0,1,\dots,n-2\big\}\right). 
\end{align*}
\end{proposition}

\begin{proof}
Let $\aaa=(a_1,\dots,a_n)$ with $a_1 \geq \cdots \geq a_n$ and suppose $\beta_{i,\aaa}(I_n) \ne 0$.
Then, by Lemma \ref{2.2}(i), $a_1$ must equal to $p_k$ for some $k$.
Also, $q_2 \geq q_k$ since $x^\aaa$ must be divisible by some $x_1^{p_k} x_2^{q_k}$
and Lemma \ref{2.2}(ii) tells $a_2 \leq q_{k+1}$ (otherwise $x^\aaa$ and $x^\aaa/(x_1^{p_{k+1}}x_2^{q_{k+1}})$ have the same support).
Hence $a_2=q_k$ or $a_2=q_{k+1}$,
and Proposition \ref{2.4}
tells that $\aaa$ must be of the form either
 $\aaa=(p_k,q_k^\ell,0^{n-\ell-1})$
or $\aaa=(p_k,q_{k+1}^\ell,0^{n-\ell-1})$
for some integer $\ell$.
The desired statement follows from the following case analysis.

{\bf Case (I).}
Suppose $\aaa=(p_k,q_k^\ell,0^{n-\ell-1})$.
Then
\begin{align*}
\Delta_\aaa^{I_n} 
&= \{F \subset [\ell+1]: x_1^{p_k-1} x_2^{q_k-1} \cdots x_{\ell+1}^{q_k-1} x^{[\ell+1]\setminus F} \in I_n\}\\
&=\{ F \subset [\ell+1]: 1 \not \in F \mbox{ and } ([\ell+1]\setminus F) \cap \{2,\dots,\ell+1\} \ne \emptyset\}\\
&=\{F \subset \{2,\dots,\ell+1\}: F \ne \{2,\dots,\ell+1\}\}.
\end{align*}
Thus $\Delta_\aaa^{I_n}$ is the boundary of the $(\ell-1)$-simplex and we have
$\beta_{i,\aaa}(I_n)=\dim_\kk \widetilde H_{i-1}(\Delta_\aaa^{I_n}) \ne 0$ if and only if $i=\ell-1$, and we also have $\widetilde H_{\ell-2}(\Delta_\aaa^{I_n}) \cong \kk$.

{\bf Case (II).}
Suppose $\aaa=(p_k,q_{k+1}^\ell,0^{n-\ell-1})$.
Then
\begin{align*}
\Delta_\aaa^{I_n} 
&= \{F \subset [\ell+1]: x_1^{p_k-1} x_2^{q_{k+1}-1} \cdots x_{\ell+1}^{q_{k+1}-1} x^{[\ell+1]\setminus F} \in I_n\}\\
&=\{ F \subset [\ell+1]: 1 \not \in F \mbox{ or } ([\ell+1]\setminus F) \cap \{2,\dots,\ell+1\} \ne \emptyset\}\\
&=\{F \subset [\ell+1]: F \ne [\ell+1]\}
\end{align*}
is the boundary of the $\ell$-simplex.
Hence 
$\beta_{i,\aaa}(I_n)=\dim_\kk \widetilde H_{i-1}(\Delta_\aaa^{I_n}) \ne 0$ if and only if $i=\ell$, and we have $\widetilde H_{\ell-1}(\Delta_\aaa^{I_n}) \cong \kk$.
\end{proof}

\begin{example}
\label{2.6.1}
Consider the ideal $\LL \subset S_\infty$ generated by $(5,1)$ and $(2,2)$.
The previous proposition tells
\begin{align}
\label{2-ex}
&\big\{\big(i,(a_1,\dots,a_n)\big): \beta_{i,(a_1,\dots,a_n)}(J_n) \neq 0, a_1 \geq \cdots \geq a_n\big\}\\
\nonumber
&= \big\{ \big(i,(2,2^{i+1},0^{n-i-2})\big):i=0,1,\dots,n-2\big\}\\
\nonumber
&\hspace{16pt} \cup \big\{ \big(i,(5,1^{i+1},0^{n-i-2})\big):i=0,1,\dots,n-2\big\}\\
\nonumber
&\hspace{16pt} \cup \big\{ \big(i+1,(5,2^{i+1},0^{n-i-2})\big):i=0,1,\dots,n-2\big\}.
\end{align}
The $\ZZ$-graded version of the equation \eqref{2-ex} is nothing but the equation \eqref{0-1} in the introduction.
\end{example}

\section{Main results}

In Proposition \ref{2.6} we can see that,
when $\II$ is generated by partitions of length $2$,
if $I_n$ has an $i$th syzygy of degree $(a_1,\dots,a_n) \in \ZZ^n_{\geq 0}$ with $a_1 \geq \cdots \geq a_n \geq 1$,
then $I_{n+1}$ has an $(i+1)$th syzygy of degree $(a_1,\dots,a_n,a_n) \in \ZZ^{n+1}_{\geq 0}$.
This phenomenon actually happens to any symmetric monomial ideal.

\begin{proposition}
\label{3.3}
Let $\II \subset S_\infty$ be a symmetric monomial ideal generated by partitions of length $\leq m$, and let $I_n =\II \cap S_n$ for $n \geq 1$.
For any integer $n \geq m$ and vector $\aaa=(a_1,\dots,a_t,b,\dots,b) \in \ZZ^n_{\geq 0}$ with $a_1 \geq \cdots \geq a_t >b \geq 1$,
we have
$$\beta_{i,\aaa}(I_n) \ne 0 \ \Leftrightarrow \beta_{i+1,(\aaa,b)}(I_{n+1}) \ne 0.$$
\end{proposition}

Since the proof of the above proposition is rather technical and purely combinatorial,
we postpone its proof to the next section.
In this section,
we prove main algebraic results of this paper using Proposition \ref{3.3}.

We first prove a $\ZZ^n$-graded version of our main result.
For a symmetric monomial ideal $I \subset S_n$,
we define
$$\mathcal B(I)=\big\{\big(i,(a_1,\dots,a_n)\big)\in \ZZ_{\geq 0} \times \ZZ_{\geq 0}^n:a_1 \geq \cdots \geq a_n \geq 0, \beta_{i,(a_1,\dots,a_n)} (I) \ne 0\big\}$$
and
$$\mathcal F(I)=\big\{\big(i,(a_1,\dots,a_n)\big)\in \ZZ_{\geq 0} \times \ZZ_{\geq 0}^n:a_1 \geq \cdots \geq a_n\geq 1, \beta_{i,(a_1,\dots,a_n)} (I) \ne 0\big\}.$$
Since the action of $S_n$ to $\Tor_i(I,\kk)$ permutes the $\ZZ^n$-grading of its elements,
the set $\mathcal B(I)$ determines the set
$\{(i,\aaa) \in \ZZ_{\geq 0} \times \ZZ^n_{\geq 0}:\beta_{i,\aaa}(I) \ne 0\}$
of nonzero positions of the $\ZZ^n$-graded Betti table of $I$.
Also, since Lemma \ref{2.3} tells 
$$\beta_{i,(a_1,\dots,a_{n-1},0)}(I) \ne 0
\Leftrightarrow
\beta_{i,(a_1,\dots,a_{n-1})}(I\cap S_{n-1}) \ne 0,$$
we have
\begin{align}
\label{3-3a}
\mathcal B(I)= \bigcup_{t=1}^n \big\{ \big(i,(a_1,\dots,a_t,0^{n-t})\big)\in \ZZ_{\geq 0} \times \ZZ^n_{\geq 0}:
\big(i,(a_1,\dots,a_t)\big) \in \mathcal F(I\cap S_t)\big\}.
\end{align}

Now consider a symmetric monomial ideal $\II \subset S_\infty$ and let $I_n=\II \cap S_n$ for  $n \geq 1$.
The equation \eqref{3-3a} tells that, for any integer $n \geq 1$, we have
$$\mathcal B(I_n)= \bigcup_{t=1}^n \big\{ \big(i,(a_1,\dots,a_t,0^{n-t})\big) : \big(i,(a_1,\dots,a_t)\big) \in \mathcal F(I_t)\big\}.$$
Moreover, Propositions \ref{2.4} and \ref{3.3} say that if $\II$ is generated by partitions of length $\leq m$, then we have
$$\mathcal F(I_{n+1})=\big\{\big(i+1,(a_1,\dots,a_n,a_n)\big)\in \ZZ_{\geq 0} \times \ZZ_{\geq 0}^{n+1}: \big(i,(a_1,\dots,a_n)\big) \in \mathcal F(I_n)\big\}$$
for $n \geq m$.
In particular, $\mathcal F(I_m)$ determines $\mathcal F(I_n)$ for all $n \geq m$.
These equations tell that the set $\mathcal B(I_n)$ is determined from the set $\mathcal F(I_1),\dots,\mathcal F(I_m)$ in the following form.

\begin{theorem}
\label{3.4}
Let $\II \subset S_\infty$ be a symmetric monomial ideal generated by partitions of length $\leq m$, and let $I_n=\II \cap S_n$ for $n \geq 1$.
Then, for any integer $n \geq m$,
\begin{align*}
\mathcal B(I_n) = &
\left( \bigcup_{t=1}^{m-1} \big\{ \big(i,(a_1,\dots,a_t,0^{n-t})\big) : \big(i,(a_1,\dots,a_t)\big) \in \mathcal F(I_t)\big\} \right)\\
 &\cup \!
\left( \bigcup_{k=m}^{n}\! \big\{ \big(i\!+\!k\!-\!m,(a_1, \dots,a_m,a_m^{k-m},0^{n-k})\big)\! :\! \big(i,(a_1,\dots,a_m)\big)\! \in\! \mathcal F(I_m)\big\}\! \right).
\end{align*}
\end{theorem}

\begin{example}
\label{3.5}
Consider the ideal $\LL$ generated by partitions $(5,1)$ and $(2,2)$ in the introduction.
In this case, we have
\begin{align*}
\mathcal F(J_1)&=\emptyset,\\
\mathcal F(J_2)&=\big\{\big(0,(2,2)\big),\big(0,(5,1)\big),\big(1,(5,2)\big)\big\},\\
\mathcal F(J_3)&=\big\{\big(1,(2,2,2)\big),\big(1,(5,1,1)\big),\big(2,(5,2,2)\big)\big\},\\
\mathcal F(J_4)&=\big\{\big(2,(2,2,2,2)\big),\big(2,(5,1,1,1)\big),\big(3,(5,2,2,2)\big)\big\},\\
&\  \ \vdots
\end{align*}
Also, the set $\mathcal B(J_n)$ is essentially a union of $\mathcal F(J_1),\mathcal F(J_2),\dots,\mathcal F(J_n)$ (if we ignore zeros in degrees). For example, when $n=4$ we have
$$
\mathcal B(J_4)=
\left\{
\begin{array}{cccccc}
\big(0,(2,2,0,0)\big),\big(0,(5,1,0,0)\big),\big(1,(5,2,0,0)\big),\\
\big(1,(2,2,2,0)\big),\big(1,(5,1,1,0)\big),\big(2,(5,2,2,0)\big),\\
\big(2,(2,2,2,2)\big),\big(2,(5,1,1,1)\big),\big(3,(5,2,2,2)\big)
\end{array}\right\}.$$
\end{example}

Now Theorem \ref{1.2} in the introduction is an easy consequence of the above theorem.

\begin{proof}[Proof of Theorem \ref{1.2}]
Theorem \ref{3.4} says
\begin{align*}
&\{(i,j): \beta_{i,i+j}(I_n) \ne 0\}\\
&= \big\{ (i,|\aaa| -i): (i,\aaa) \in \mathcal B(I_n)\big\}\\
&=\bigcup_{t=1}^{m-1} \big\{(i,|\aaa|-i): (i,\aaa) \in \mathcal F(I_t)\big\}\\
&\hspace{16pt}
\cup\!\!\!\! \bigcup_{(i,(a_1,\dots,a_m)) \in \mathcal F(I_m)} \!\!\!\!\!\! \big\{ (i\!+\!k\!-\!m, a_1\!+\! \cdots\! +\!a_m+a_m(k\!-\!m)\!-\!i\!-\!k\!+\!m): m \leq k \leq n\big\}\\
&= \big\{(i,j): \beta_{i,i+j}(I_{m-1})\ne 0\big\}\\
&\hspace{16pt} \cup \bigcup_{(i,(a_1,\dots,a_m)) \in \mathcal F(I_m)} \big\{ (i+k, a_1\!+\! \cdots\! +\!a_m-i+(a_m-1)k): 0 \leq k \leq n-m\big\},
\end{align*}
where we use $\mathcal B(I_{m-1})= \bigcup_{t=1}^{m-1} \big\{ \big(i,(\aaa,0^{m-1-t})\big): (i,\aaa) \in \mathcal F(I_t)\big\}$
for the last equality.
By setting 
$$D=\big\{(i,a_1+\cdots+a_m-i,a_m-1):\big(i,(a_1,\dots,a_m)\big) \in \mathcal F(I_m)\big\},$$
we get the desired statement.
\end{proof}

Another formulation of Theorem \ref{3.4} is
{\small
\begin{align*}
\mathcal B(I_n) = &
\big\{ \big(i,(\aaa,0^{n-m})\big): (i,\aaa) \in \mathcal B(I_m) \setminus \mathcal F(I_m) \big\}\\
 &\cup
\left( \bigcup_{k=m}^{n} \big\{ \big(i+k-m,(a_1,\dots,a_m,a_m^{k-m},0^{n-k})\big) : \big(i,(a_1,\dots,a_m)\big) \in \mathcal F(I_m)\big\} \right).
\end{align*}
}This formula enable us to 
determine the shape of the Betti table of $I_n$ from $\ZZ^n$-graded Betti numbers of $I_m$.
Below we give two examples.

\begin{example}
\label{3.6}
Let $m \geq 1$ be an integer and let $\mathcal T \subset S_\infty$ be the symmetric monomial ideal
generated by $m$ partitions
$$(1^m),(2^{m-1}),(3^{m-2}),\dots,(m^1).$$
Let $T_n=\mathcal T \cap S_n$ for $n \geq 1$.
The ideal $T_m$ is called a tree ideal,
and it is known that the Scarf complex gives a minimal free resolution of $T_m$ (see \cite[Example 1.2 and Theorem 1.5]{MSY}).
This tells that
the set $\{(i,\aaa):\beta_{i,\aaa}(T_m) \ne 0\}$ equals to
\begin{align*}
\{ (|F|\!-\!1,\mathrm{lcm}(F))\!:\! F \!\subset\! G(T_m),\ \mathrm{lcm}(F) \!\ne\! \mathrm{lcm}(F') \mbox{ for any } F'\! \subset\! G(I) \mbox{ with } F'\! \ne\! F \}.
\end{align*}
When $n >m$,
a minimal free resolution of $T_n$ cannot be given by the Scarf complex.
However, one can determine the shape of its Betti table by using Theorem \ref{3.4}.
For example, when $m=4$ we have
{\small
\begin{align*}
\mathcal B(T_4)
= \left\{
\begin{array}{lll}
\big(0,(4,0,0,0)\big),\big(0,(3,3,0,0)\big),\big(0,(2,2,2,0)\big),\big(0,(1,1,1,1)\big),\\
\big(1,(4,3,0,0)\big),\big(1,(4,2,2,0)\big),\big(1,(4,1,1,1)\big),\\
\big(1,(3,3,2,0)\big),\big(1,(3,3,1,1)\big),\big(1,(2,2,2,1)\big),\\
\big(2,(4,3,2,0)\big),\big(2,(4,3,1,1)\big),\big(2,(4,2,2,1)\big),\big(2,(3,3,2,1)\big),\\\big(3,(4,3,2,1)\big)
\end{array}
\right\}
\end{align*} 
}We have $7$ elements in $\mathcal B(T_4) \setminus \mathcal F(T_4)$ each of which contributes a single position in the Betti table of $I_n$, and we have $8$ elements in $\mathcal F(T_4)$ each of which creates a line segment of length $n-4$ in the Betti table.
See Figure \ref{fig4}.

\begin{figure}[h]
\includegraphics[scale=0.28]{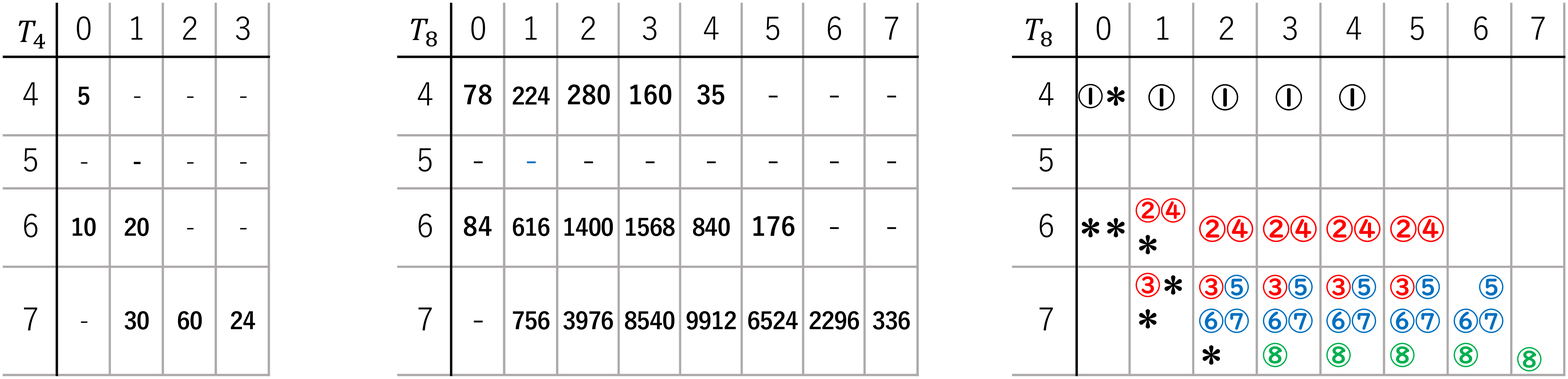}
\caption{Betti tables of $T_4$ and $T_8$. Each numbered circle represents a line segment created by an element of $\mathcal F(T_4)$. The positions where elements of $\mathcal B(T_4)\setminus \mathcal F(T_4)$ contribute are marked by $*$.}
\label{fig4}
\end{figure}

\end{example}

\begin{example}
\label{3.7}
The situation becomes simpler when $\II$ is generated by partitions of the same length $m$ since in such a case we have $\mathcal B(I_m)=\mathcal F(I_m)$.
For example, let $\mathcal P \subset S_\infty$ be the symmetric monomial ideal generated by a single partition $(m,m-1,\dots,2,1)$ and let $P_n=\mathcal P \cap S_n$ for $n \geq 1$.
The ideal $P_m$ is known as a permutohedron ideal and its minimal free resolution is given by the co-Scarf complex (see \cite[Example 4.2 and Theorem 4.6]{MSY}), which tells
\[
\mathcal B(P_m)=\big \{ \big(m-i-1, (m,\dots,m)-\aaa^* \big) : \aaa \in \mathcal F(T_m)\big\},
\]
where $(a_1,a_2,\dots,a_n)^*=(a_n,\dots,a_2,a_1)$.
For example, when $m=4$, we have
{\small
\begin{align*}
\mathcal B(\mathcal P_4)
= \left\{
\begin{array}{lll}
\big(0,(4,3,2,1)\big),
\big(1,(4,4,2,1)\big),\big(1,(4,3,3,1)\big),\big(1,(4,3,2,2)\big),\\
\big(2,(4,4,4,1)\big),\big(2,(4,4,2,2)\big),\big(2,(4,3,3,3)\big),\big(3,(4,4,4,4)\big)
\end{array}
\right\}.
\end{align*}
}Theorem \ref{3.4} tells that the Betti table of $P_n$ is a union of $8$ line segments of length $n-4$ created by $8$ elements in $\mathcal F(P_4)=\mathcal B(P_4)$.
See Figure \ref{fig5}.
\begin{figure}
\includegraphics[scale=0.42]{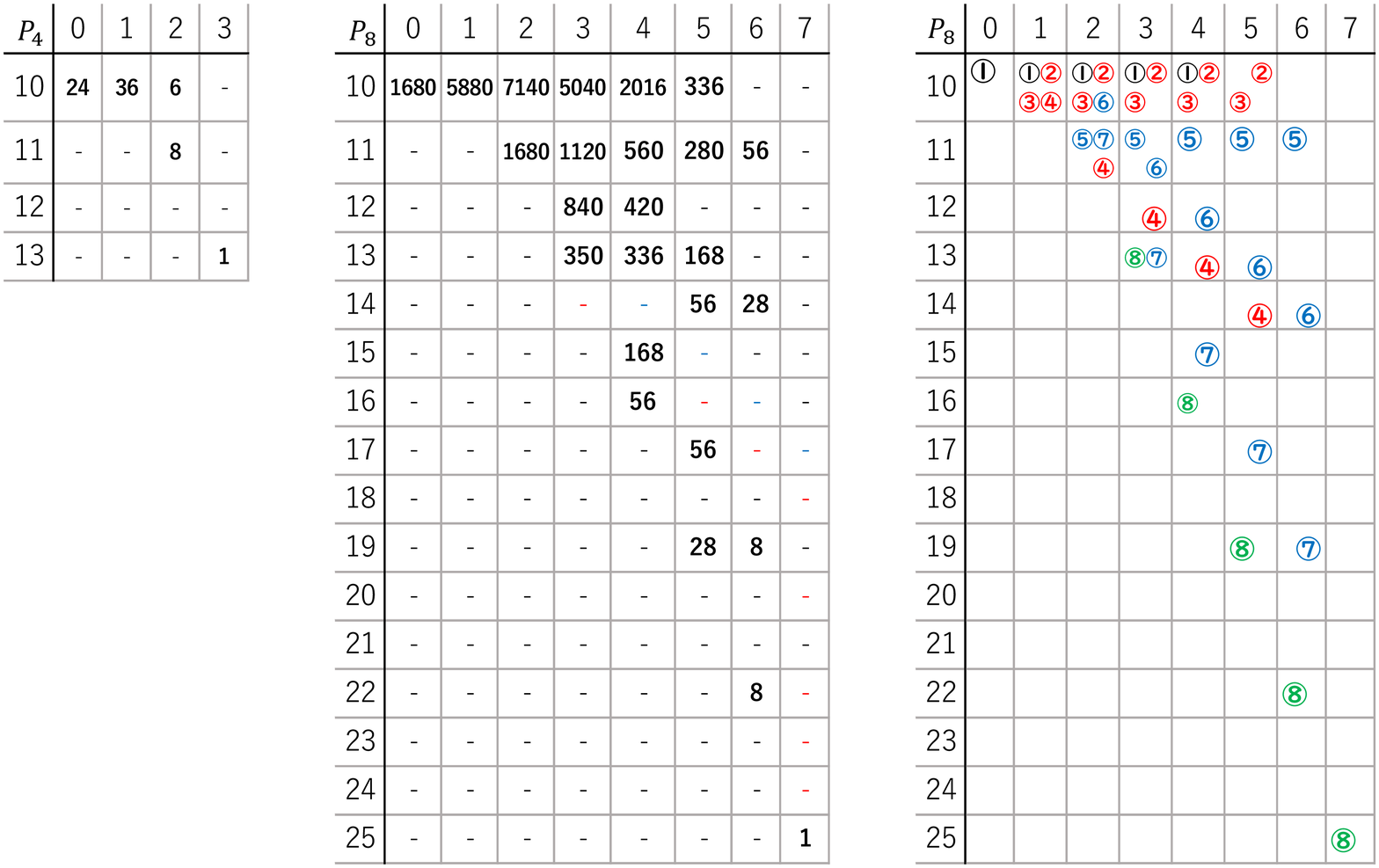}
\caption{Betti tables of $P_4$ and $P_8$. Numbered circles represent $8$ line segments created by elements of $\mathcal B(P_4)=\mathcal F(P_4)$.}
\label{fig5}
\end{figure}
\end{example}

Finally,
we discuss projective dimension and regularities of symmetric monomial ideals.
Recall that, for a homogeneous ideal $I \subset S_n$, the {\bf projective dimension} of $I$ is the number
$$\pd(I)=\max\{i: \beta_{i,j}(I) \ne 0 \mbox{ for some $j$}\}$$
and the {\bf Castelnuovo-Mumford regularity} of $I$ is the number
$$\reg(I)=\max\{j: \beta_{i,i+j}(I) \ne 0 \mbox{ for some $i$}\}.$$
We first prove the following statement which is analogous to Proposition \ref{3.3} but does not have any assumption on length of generators.

\begin{proposition}
\label{3.8}
Let $\II \subset S_\infty$ be a symmetric monomial ideal and $I_n = \II \cap S_n$ for $n \geq 1$.
Let $\aaa=(a_1,\dots,a_n) \in \ZZ^n_{\geq 0}$ with $a_1 \geq a_2 \geq \cdots \geq a_n \geq 1$.
If $\beta_{i,\aaa}(I_n) \ne 0$ then $\beta_{i+1,(\aaa,k)}(I_{n+1}) \ne 0$ for some $1\leq k \leq a_n$.
\end{proposition}

\begin{proof}
Let 
$$\Delta^{(k)}= \{F \subset [n]: x^\ua \cdot x_{n+1}^k \cdot  x^{[n] \setminus F} \in I_{n+1}\}$$
for $k=0,1,\dots,a_n$
and $\Delta=\Delta^{(0)}$.
Also, let $\Sigma=\Delta \cup \{\{n\} \cup F: F \in \Delta\}$.
The definition tells
\begin{itemize}
\item[(a)] $\Delta=\Delta^{(0)}=\Delta_\aaa^{I_n}$.
\item[(b)] $\Delta^{(0)} \subset \Delta^{(1)} \subset \cdots \subset \Delta^{(a_n)}$.
\item[(c)]  $\Delta_{(\aaa,k)}^{I_{n+1}}= \Delta^{(k)} \cup \{ \{n+1\} \cup F: F \in \Delta^{(k-1)}\}$.
\end{itemize}
Also, we have the inclusion
\begin{itemize}
\item[(d)] $\Delta^{(a_n)} \supset \Sigma$.
\end{itemize}
To see (d), it is enough to show that $F \cup \{n\} \in \Delta^{(a_n)}$ for any $F \in \Delta$ with $n \not \in F$.
Indeed, if $F \in \Delta$ with $n \not \in F$,
then $x^\ua \cdot x^{[n] \setminus F} \in I_n$
and we get
$$ x^\ua \cdot x^{[n] \setminus F} \cdot x_{n+1}^{a_n-1}
=x_1^{a_1-1} \cdots x_{n-1}^{a_{n-1}-1} \cdot  x^{[n-1] \setminus F} \cdot x_n^{a_n} x_{n+1}^{a_n-1} \in I_{n+1}.$$
Then by exchanging $x_n$ and $x_{n+1}$ using the symmetry of $I_{n+1}$ we have
$$ x^\ua \cdot x^{[n] \setminus (F \cup \{n\})} \cdot x_{n+1}^{a_n}
=x_1^{a_1-1} \cdots x_{n-1}^{a_{n-1}-1} \cdot x^{[n-1] \setminus F} \cdot x_n^{a_n-1} x_{n+1}^{a_n} \in I_{n+1},$$
which implies $F \cup \{n\} \in \Delta^{(a_n)}$.

Recall that we assume $\widetilde H_{i-1}(\Delta) \cong \Tor_i(I_n,\kk)_{\aaa} \ne 0$.
Take a nonzero element $\alpha \in \widetilde H_{i-1}(\Delta)$.
Consider the maps
$$\widetilde H_{i-1}(\Delta^{(0)}) \stackrel {\iota_1}
\longrightarrow
\widetilde H_{i-1}(\Delta^{(1)}) \stackrel {\iota_2}
\longrightarrow
\cdots \stackrel {\iota_{a_n}}
\longrightarrow
\widetilde H_{i-1}(\Delta^{(a_n)})$$
induced by the inclusion (b),
and let $\alpha_i=\iota_i\circ \cdots \circ \iota_1(\alpha)$.
Then (d) and Lemma \ref{3.1} say that the composition $\iota_{a_n} \circ \cdots \circ \iota_{1}$ is a zero map, so $\alpha_{a_n}=0$.
Hence there is an integer $ 1 \leq k \leq a_n$ such that
$\alpha_k=0$ but
$\alpha_{k-1}$ is nonzero, where $\alpha_0=\alpha$.
On the other hand, the equation (c) gives a long exact sequence of the pair $(\Delta^{I_{n+1}}_{(\aaa,k)},\Delta^{(k)})$ (see e.g.\ \cite[\S4.5 ]{Sp})
\begin{align*}
\cdots
\longrightarrow
\widetilde H_{i}(\Delta_{(\aaa,k)}^{I_{n+1}})
\stackrel {\eta} \longrightarrow
H_{i}(\Delta_{(\aaa,k)}^{I_{n+1}},\Delta^{(k)})
\cong \widetilde H_{i-1}(\Delta^{(k-1)})
\stackrel {\rho} \longrightarrow
\widetilde H_{i-1}(\Delta^{(k)})
\longrightarrow\cdots.
\end{align*}
By a routine diagram chase computation, we can see that the image of $\alpha_{k-1} \in \widetilde H_{i-1}(\Delta^{(k-1)})$ by the map $\rho$ is nothing but $\pm \alpha_k=\pm \iota_k(\alpha_{k-1})$,
so $\alpha_{k-1}$ is contained in the kernel of $\rho$.
Hence $\eta$ is not a zero map
and we get $\beta_{i+1,(\aaa,k)}(I_{n+1}) = \dim _\kk \widetilde H_i(\Delta^{I_{n+1}}_{(\aaa,k)}) \ne 0$.
\end{proof}

Let $\II \subset S_\infty$ be a nonzero proper symmetric monomial ideal generated by partitions of length $\leq m$, and let $I_n=\II \cap S_n$ for $n \geq 1$.
Theorem \ref{1.2} says that $\pd(I_{n+1}) \leq \pd(I_n)+1$ for $n \geq m$.
On the other hand, Proposition \ref{3.8} tells that $\beta_{i,(a_1,\dots,a_t,0,\dots,0)}(I_n) \ne 0$ with $a_1,\dots,a_t \ne 0$ implies $\beta_{i,(a_1,\dots,a_t,k,0,\dots,0)}(I_{n+1}) \ne 0$ for some $k \geq 1$ (use Lemma \ref{2.3}(i) when $t \ne n$).
This in particular shows $\pd(I_{n+1}) \geq \pd (I_n)+1$ when $I_n \ne \{0\}$ and $I_n \ne S_n$.
Hence we get the next statement.

\begin{corollary}
\label{3.9}
With the same notation as above, $\pd(I_{n+1})=\pd(I_n)+1$ for $n \geq m$.
\end{corollary}

The above corollary also tells that for a symmetric monomial ideal $I \subset S_n$,
the Cohen--Macaulay property of $S_n/I$ only depends on the set $\Lambda(I)$.
Recall that, for a homogeneous ideal $I \subset S_n$,
the {\bf depth} of $S_n/I$ is the number $\mathrm{depth}(S_n/I)=n-1-\pd(I)$,
and we say that $S_n/I$ is {\bf Cohen--Macaulay} if the Krull dimension of $S_n/I$ is equal to its depth.

\begin{corollary}
\label{3.10}
Let $I \subset S_n$ and $J \subset S_m$ be symmetric monomial ideals with $\Lambda(I)=\Lambda(J)$.
Then $S_n/I$ is Cohen--Macaulay if and only if $S_m/J$ is Cohen--Macaulay.
\end{corollary}

\begin{proof}
Corollary \ref{3.9} tells that
the depth of $S_n/I$ only depends on $\Lambda(I)$.
Thus it is enough to check that the Krull dimension of $S_n/I$ only depends on $\Lambda(I)$.
Let $r$ be the smallest length of partitions in $\Lambda(I)$.
Then the radical $\sqrt I$ of $I$ is the symmetric monomial ideal generated by a single partition $(1^r)$,
in other words,
$\sqrt I$ is the ideal generated by all squarefree monomials of degree $r$.
Thus the Krull dimension of $S_n/\sqrt I$ is equal to $r-1$ (see \cite[II Theorem 1.3]{St}),
which must coincide with the Krull dimension of $S_n/I$.
Thus the Krull dimension of $S_n/I$ only depends on $\Lambda(I)$.
\end{proof}

Next, we discuss Castelnuovo-Mumford regularity.
The following result is a more precise version of Corollary \ref{1.3} and solves a special case of \cite[Conjecture 3.8]{LNNR1}.

\begin{proposition}
\label{3.11}
Let $\II \subset S_\infty$ be a symmetric monomial ideal and let $I_n=\II \cap S_n$ for $n \geq 1$.
Let $w= \min\{\lambda_1: (\lambda_1,\dots,\lambda_k) \in \Lambda(\II)\}$.
There is an integer $C>0$ such that
\[
\reg(I_n)= (w-1)n +C \ \ \ \mbox{ for }n \gg 0.
\]
\end{proposition}

\begin{proof}
Suppose that $\II$ is generated by partitions of length $\leq m$.
Theorem \ref{3.4} tells that there are integers $W,C$ such that
$\reg(I_n)=Wn+C$ for $n \gg 0$,
and
\begin{align}
\label{3-9-1}
W=\max\{a_m:\beta_{i,(a_1,\dots,a_m)}(I_m) \ne 0 \mbox{ for some $i$ and $a_1 \geq \cdots \geq a_m \geq 1$}\}-1.
\end{align}

We will prove $W=w-1$.
We first show $W \leq w-1$ (this fact was actually proved in \cite[Corollary 3.5]{LNNR1} in more general setting, but we include its proof).
Since $\Lambda(I_n)$ contains a partition $\lambda=(\lambda_1,\dots,\lambda_k)$
with $w \geq \lambda_1 \geq \cdots \geq \lambda_k$ and $k \leq m$,
any vector $\aaa=(a_1,\dots,a_m)$ with
$a_1 \geq \cdots \geq a_m \geq w+1$ satisfies $\supp(x^\aaa/x^\lambda)=\supp(x^\aaa/x^\lambda)$.
Thus by Lemma \ref{2.2}(ii) we have
$\beta_{i,\aaa}(I_m) =0$ for all $i$ if $a_1 \geq \cdots \geq a_m \geq w+1$.
Hence $W \leq w-1$ by \eqref{3-9-1}.

We finally prove $W \geq w-1$.
Let 
$p=\min\{k:x_1^w \cdots x_k^w x_{k+1}^{w-1} \cdots x_m^{w-1} \in I_m\}.$
Then
\begin{align*}
\Delta^{I_m}_{(w,\dots,w)} 
&= \{ F \subset [m]: x_1^{w-1} \cdots x_m^{w-1} x^{[m] \setminus F} \in I_m\}
= \{ F \subset [m]: |F| \leq m-p\}.
\end{align*}
This simplicial complex clearly has a nontrivial element in homological position $m-p-1$.
Hence
\[
\beta_{m-p,(w,\dots,w)}(I_m) = \dim_\kk \widetilde H_{m-p-1}(\Delta^{I_m}_{(w,\dots,w)}) \ne 0.
\]
By \eqref{3-9-1}, this proves $W \geq w-1$ as desired.
\end{proof}

Since, in Proposition \ref{3.11}, $w=1$ is equivalent to the condition that $\II$ contains a squarefree monomial,
we get the following simple characterization of symmetric monomial ideals whose regularity becomes constant.
Note that the if part was proved in \cite[Proposition 3.9]{LNNR1} in more general setting.

\begin{corollary}
\label{constant}
Let $\II \subset S_\infty$ be a symmetric monomial ideal generated by partitions of length $\leq m$ and let $I_n=\II \cap S_n$ for $ n \geq 1$.
Then $\II$ contains a squarefree monomial if and only if $\reg(I_n)$ is constant for $n \gg 0$.
Moreover, if $\reg(I_n)$ is constant for $n \gg 0$, then
$\reg(I_n)=\reg(I_m)$ for all $n \geq m$.
\end{corollary}

\begin{proof}
The first statement is an immediate consequence of Corollary \ref{3.11}.
The second statement follows form Theorem \ref{1.2} since if $\reg(I_n)$ is constant for $n \gg 0$, then any element $(i,j,c)$ in the set $D$ in Theorem \ref{1.2} must satisfies $c=0$.
\end{proof}

The ideal $\mathcal T$ in Example \ref{3.6} contains a squarefree monomial ideal
so the regularity of $T_n$ is constant for $ n\geq m$.

\begin{remark}
It was pointed out by Claudiu Raicu that the number $C$ in Proposition \ref{3.11} can be determined as follows: Let $\II$, $m$ and $w$ be as in the proposition, and let $\alpha_n=(x_1 \cdots x_n)^{w-1}$ for $n\geq 1$. Then
\begin{align}
\label{claudiu}
\reg(I_n)=(w-1)n+\reg(I_m:\alpha_m)\ \ \mbox{for $n \gg 0$}.
\end{align}

To see this, consider the short exact sequence
$$
0 \longrightarrow S_n/(I_n:\alpha_n) \stackrel {\times \alpha_n} \longrightarrow S_n/I_n \longrightarrow S_n/(I_n+(\alpha_n)) \longrightarrow 0.
$$
Then $\Tor_i(S_n/(I_n+(\alpha_n)),\kk)_{\bb}=0$
for any $\bb \geq (w,w,\dots,w)$ (as the ideal contains $\alpha_n=x_1^{w-1}x_2^{w-1}\cdots x_n^{w-1}$ Lemma \ref{2.2}(i) implies this).
Also,
the regularity of $I_n$ must be attained by some multigraded component of $\Tor_i(I_n,\kk)$ of degree $(a_1,\dots,a_{m-1},w,\dots,w)$ for $n \gg 0$.
These two facts tell $\reg(I_n)=(w-1)n+\reg(I_n:\alpha_n)$ for $n \gg 0$. Since $I_n:\alpha_n$ contains a squarefree monomial for $n \geq m$, Corollary \ref{constant} tells $\reg(I_n:\alpha_n)=\reg(I_m:\alpha_m)$ for $n \geq m$,
which guarantees \eqref{claudiu}.
\end{remark}

\section{Proof of Proposition \ref{3.3}}

In this section, we prove our key proposition Proposition \ref{3.3}. For a simplicial complex $\Delta$ and an element $v$ ($v$ need not to be a vertex of $\Delta$),
we define the simplicial complex $\Delta^{[v]}$ by
$$\Delta^{[v]}= \Delta \cup \{F \cup \{v\}: F \in \Delta\}.$$
This simplicial complex is always a cone with apex $v$,
so it is an acyclic simplicial complex.
See Figure \ref{fig3}.
\bigskip

\begin{figure}[h]
{\unitlength 0.1in%
\begin{picture}(57.0000,9.1000)(13.6000,-18.4000)%
%
\special{pn 8}%
\special{pa 1600 1000}%
\special{pa 1600 1600}%
\special{fp}%
\special{pa 1600 1600}%
\special{pa 1600 1600}%
\special{fp}%
%
\special{sh 0}%
\special{ia 1600 1600 50 50 0.0000000 6.2831853}%
\special{pn 8}%
\special{ar 1600 1600 50 50 0.0000000 6.2831853}%
%
\special{sh 0}%
\special{ia 2200 1600 50 50 0.0000000 6.2831853}%
\special{pn 8}%
\special{ar 2200 1600 50 50 0.0000000 6.2831853}%
%
\special{sh 0}%
\special{ia 2200 1000 50 50 0.0000000 6.2831853}%
\special{pn 8}%
\special{ar 2200 1000 50 50 0.0000000 6.2831853}%
%
\special{sh 0}%
\special{ia 1600 1000 50 50 0.0000000 6.2831853}%
\special{pn 8}%
\special{ar 1600 1000 50 50 0.0000000 6.2831853}%
\put(14.5000,-16.0000){\makebox(0,0){$1$}}%
\put(14.5000,-10.0000){\makebox(0,0){$2$}}%
\put(23.5000,-16.0000){\makebox(0,0){$3$}}%
\put(23.5000,-10.0000){\makebox(0,0){$4$}}%
%
\special{pn 8}%
\special{pa 3200 995}%
\special{pa 3200 1595}%
\special{fp}%
\special{pa 3200 1595}%
\special{pa 3200 1595}%
\special{fp}%
%
\special{sh 0}%
\special{ia 3200 995 50 50 0.0000000 6.2831853}%
\special{pn 8}%
\special{ar 3200 995 50 50 0.0000000 6.2831853}%
\put(30.5000,-15.9500){\makebox(0,0){$1$}}%
\put(30.5000,-9.9500){\makebox(0,0){$2$}}%
\put(39.5000,-15.9500){\makebox(0,0){$3$}}%
\put(39.5000,-9.9500){\makebox(0,0){$4$}}%
%
\special{pn 8}%
\special{pa 4800 995}%
\special{pa 4800 1595}%
\special{fp}%
\special{pa 4800 1595}%
\special{pa 4800 1595}%
\special{fp}%
\put(46.5000,-15.9500){\makebox(0,0){$1$}}%
\put(46.5000,-9.9500){\makebox(0,0){$2$}}%
\put(55.5000,-15.9500){\makebox(0,0){$3$}}%
\put(55.5000,-9.9500){\makebox(0,0){$4$}}%
%
\special{pn 8}%
\special{pa 6400 995}%
\special{pa 6400 1595}%
\special{fp}%
\special{pa 6400 1595}%
\special{pa 6400 1595}%
\special{fp}%
\put(62.5000,-9.9500){\makebox(0,0){$2$}}%
\put(71.5000,-15.9500){\makebox(0,0){$3$}}%
\put(71.5000,-9.9500){\makebox(0,0){$4$}}%
%
\special{pn 8}%
\special{pa 3200 1600}%
\special{pa 3800 1600}%
\special{fp}%
\special{pa 3200 1600}%
\special{pa 3800 1000}%
\special{fp}%
%
\special{pn 8}%
\special{pa 6400 1600}%
\special{pa 6700 1300}%
\special{fp}%
\special{pa 6700 1300}%
\special{pa 6400 1000}%
\special{fp}%
\special{pa 7000 1000}%
\special{pa 6700 1300}%
\special{fp}%
\special{pa 6700 1300}%
\special{pa 7000 1600}%
\special{fp}%
%
\special{pn 4}%
\special{pa 5180 1380}%
\special{pa 4960 1600}%
\special{fp}%
\special{pa 5120 1320}%
\special{pa 4850 1590}%
\special{fp}%
\special{pa 5060 1260}%
\special{pa 4800 1520}%
\special{fp}%
\special{pa 5000 1200}%
\special{pa 4800 1400}%
\special{fp}%
\special{pa 4940 1140}%
\special{pa 4800 1280}%
\special{fp}%
\special{pa 4880 1080}%
\special{pa 4800 1160}%
\special{fp}%
\special{pa 5240 1440}%
\special{pa 5080 1600}%
\special{fp}%
\special{pa 5300 1500}%
\special{pa 5200 1600}%
\special{fp}%
\special{pa 5360 1560}%
\special{pa 5320 1600}%
\special{fp}%
%
\special{pn 4}%
\special{pa 6660 1260}%
\special{pa 6400 1520}%
\special{fp}%
\special{pa 6600 1200}%
\special{pa 6400 1400}%
\special{fp}%
\special{pa 6540 1140}%
\special{pa 6400 1280}%
\special{fp}%
\special{pa 6480 1080}%
\special{pa 6400 1160}%
\special{fp}%
%
\special{sh 0}%
\special{ia 3200 1595 50 50 0.0000000 6.2831853}%
\special{pn 8}%
\special{ar 3200 1595 50 50 0.0000000 6.2831853}%
%
\special{sh 0}%
\special{ia 3800 1595 50 50 0.0000000 6.2831853}%
\special{pn 8}%
\special{ar 3800 1595 50 50 0.0000000 6.2831853}%
%
\special{sh 0}%
\special{ia 3800 995 50 50 0.0000000 6.2831853}%
\special{pn 8}%
\special{ar 3800 995 50 50 0.0000000 6.2831853}%
%
\special{pn 8}%
\special{pa 4800 1595}%
\special{pa 5400 1595}%
\special{fp}%
\special{pa 5400 1595}%
\special{pa 5400 995}%
\special{fp}%
\special{pa 5400 1595}%
\special{pa 4800 995}%
\special{fp}%
%
\special{sh 0}%
\special{ia 4800 990 50 50 0.0000000 6.2831853}%
\special{pn 8}%
\special{ar 4800 990 50 50 0.0000000 6.2831853}%
%
\special{sh 0}%
\special{ia 4800 1590 50 50 0.0000000 6.2831853}%
\special{pn 8}%
\special{ar 4800 1590 50 50 0.0000000 6.2831853}%
%
\special{sh 0}%
\special{ia 5400 1590 50 50 0.0000000 6.2831853}%
\special{pn 8}%
\special{ar 5400 1590 50 50 0.0000000 6.2831853}%
%
\special{sh 0}%
\special{ia 5400 990 50 50 0.0000000 6.2831853}%
\special{pn 8}%
\special{ar 5400 990 50 50 0.0000000 6.2831853}%
%
\special{sh 0}%
\special{ia 6400 990 50 50 0.0000000 6.2831853}%
\special{pn 8}%
\special{ar 6400 990 50 50 0.0000000 6.2831853}%
%
\special{sh 0}%
\special{ia 6400 1590 50 50 0.0000000 6.2831853}%
\special{pn 8}%
\special{ar 6400 1590 50 50 0.0000000 6.2831853}%
%
\special{sh 0}%
\special{ia 7000 1590 50 50 0.0000000 6.2831853}%
\special{pn 8}%
\special{ar 7000 1590 50 50 0.0000000 6.2831853}%
%
\special{sh 0}%
\special{ia 7000 990 50 50 0.0000000 6.2831853}%
\special{pn 8}%
\special{ar 7000 990 50 50 0.0000000 6.2831853}%
%
\special{sh 0}%
\special{ia 6700 1300 50 50 0.0000000 6.2831853}%
\special{pn 8}%
\special{ar 6700 1300 50 50 0.0000000 6.2831853}%
\put(62.5000,-15.9500){\makebox(0,0){$1$}}%
\put(67.0000,-11.5500){\makebox(0,0){$5$}}%
\put(19.2000,-19.0000){\makebox(0,0){$\Delta$}}%
\put(35.2000,-19.0500){\makebox(0,0){$\Delta^{[1]}$}}%
\put(51.2000,-19.0500){\makebox(0,0){$\Delta^{[3]}$}}%
\put(67.2000,-19.0500){\makebox(0,0){$\Delta^{[5]}$}}%
\end{picture}}%
\bigskip
\caption{Examples of $\Delta^{[v]}$}
\label{fig3}
\end{figure}
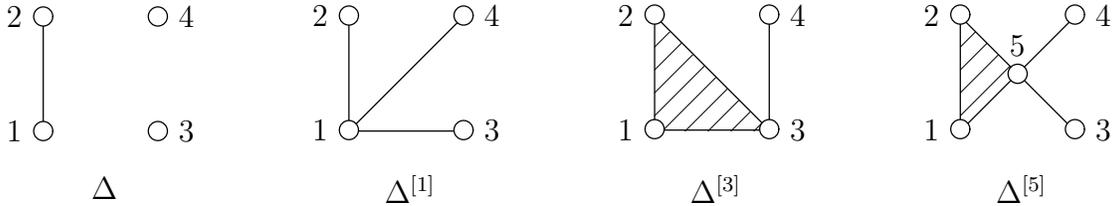

For a simplicial complex $\Delta$ on $[n]$
and a subset $X \subset [n]$,
we say that $\Delta$ is fixed by permutations on $X$
if for any permutation $\sigma \in \sym_X$ the complex
$\sigma(\Delta)=\{\sigma(F):F \in \Delta\}$ is equal to $\Delta$,
where we consider that $\sigma(k)=k$ if $k \not \in X$.
For example, the simplicial complex $\Delta$ in Figure 3 is fixed by permutations on $\{1,2\}$ (and also $\{3,4\}$).

The next lemma is quite technical and its proof is not so simple,
but it is crucial to prove Proposition \ref{3.3}.

\begin{lemma}
\label{3.2}
Let $\Delta$ be a simplicial complex on $[n]$,
$1 \leq r \leq n$ an integer,
and
$$\Sigma=\Delta^{[1]} \cup \Delta^{[2]}\cup \cdots \cup \Delta^{[r]}.$$
Suppose that $\Delta$ is fixed by permutations on $[r]$.
If $\widetilde H_i(\Sigma) \ne 0$ then $\widetilde H_{i-1}(\Delta) \ne 0$.
\end{lemma}

\begin{proof}
We prove the statement using induction on $r$.
The statement is trivial when $r=1$ since $\Delta^{[1]}$ is an acyclic simplicial complex.
Suppose $r>1$ and we assume that the statement holds for
$$\Sigma'=\Delta^{[1]} \cup \Delta^{[2]}\cup \cdots \cup \Delta^{[r-1]}.$$
Let $\Gamma=\Sigma' \cap \Delta^{[r]}$.
Note that $\Gamma$ contains $\Delta$ since each $\Delta^{[k]}$ contains $\Delta$.
We first prove the next claim.
\bigskip

\noindent
{\bf Claim.} $\dim_\kk \widetilde H_i(\Gamma) \leq \dim_\kk H_i(\Delta)$ for all $i$.

\begin{proof}[Proof of Claim]
Let
$$\Delta-[r]=\{ F \in \Delta: F \subset \{r+1,\dots,n\} \}$$
and 
$$\st(i)=\{F \in \Delta: \{i\} \cup F \in \Delta\}$$
for $i=1,2,\dots,r$.
The definition tells that each $\st(i)$ is a cone with apex $i$ and we have
$$\Delta=(\Delta-[r]) \cup \left(\bigcup_{k=1}^r \st(k)\right).$$
We will prove
\begin{align}
\label{3-1}
\Gamma= (\Delta-[r]) \cup \left( \bigcup_{k=1}^r \st(k) \right)^{[r]}.
\end{align}

We first prove the inclusion ``$\supset$'' in \eqref{3-1}.
Recall that $\Gamma = \Sigma'\cup \Delta^{[r]}$.
Since $\Delta \supset \bigcup_{k=1}^r \st(k)$,
the inclusion $\Delta^{[r]} \supset  (\Delta-[r]) \cup \left( \bigcup_{k=1}^r \st(k) \right)^{[r]}$ is clear.
It is enough to prove that any element $F \in \left(\bigcup_{k=1}^r \st(k)\right)^{[r]}$ with $F \not \in \Delta$ is contained in $\Sigma'$.
Take such an element $F$.
Then $F$ must contain $r$ since $F \not \in \Delta$ and we have $F \setminus \{r\} \in \bigcup_{k=1}^r \st(k)$.
Suppose $F \setminus \{r\} \in \st(\ell)$ with $1 \leq \ell \leq r$.
If $\ell \not \in F\setminus \{r\}$,
then $(F \setminus \{r\}) \cup \{\ell\} \in \st_\Delta(\ell) \subset \Delta$,
and since $\Delta$ is fixed by permutations on $[r]$ by exchanging $r$ and $\ell$ in $(F \setminus \{r\}) \cup \{\ell\}$
(if $r \ne \ell$) we have $F=( F \setminus \{r\}) \cup \{r\} \in \Delta \subset \Sigma'$.
On the other hand,
if $\ell \in F \setminus \{r\}$,
then $\ell <r$ and by exchanging $\ell$ and $r$ in $F \setminus \{r\} \in \Delta$,
we have $F \setminus \{\ell \} \in \Delta$ and therefore $F \in \Delta^{[\ell]} \subset \Sigma'$ as desired.

Next, we prove the inclusion `$\subset$' in \eqref{3-1}.
It is enough to prove that any element $F \in \Gamma$ with $F \not \in \Delta$ is contained in $\left(\bigcup_{k=1}^r \st(k)\right)^{[r]}$.
Take such an $F \in \Gamma$ and suppose $F \in \Delta ^{[\ell]}$ with $\ell <r$.
Since $F \in \Delta^{[\ell]} \cap \Delta^{[r]}$ and $F \not \in \Delta$,
$F$ must contain both $\ell$ and $r$. Also, we have $F \setminus \{r\} \in \Delta$,
and since $\ell \in F \setminus \{r\} \in \Delta$
we have $F \setminus \{r\} \in \st(\ell)$, which tells $F \in \st(\ell)^{[r]}$ as desired.

We now complete the proof of the claim.
Let $X=\bigcup_{k=1}^r \st(k)$
and $Y=(\Delta-[r])\cap X$.
Any element $F \in Y$ satisfies $F \subset \{r+1,\dots,n\}$ and $F \cup \{k\} \in \Delta$ for some $1 \leq k\leq r$.
But $F \cup \{k\} \in \Delta$ implies $F \cup \{1\} \in \Delta$ since $\Delta$ is fixed by permutations on $[r]$. Thus any element $F \in Y$ must satisfy $F \cup \{1\} \in \Delta$, so we have
$$Y \subset \st(1).$$
Also, since any element in $\Delta-[r]$ does not contain $r$, we have
$$(\Delta-[r]) \cap X^{[r]} =Y.$$
Compare the following two Mayer--Vietris exact sequences for $\Delta=(\Delta-[r]) \cup X$ and for $\Gamma=(\Delta-[r]) \cup X^{[r]}$,
\begin{align*}
\cdots \longrightarrow
\widetilde H_i(Y) \stackrel {\delta_i=(\eta_i,\psi_i)} {\longrightarrow}
\widetilde  H_i(\Delta-[r]) \bigoplus \widetilde  H_i(X) \longrightarrow
\widetilde  H_i(\Delta) \stackrel {\rho_i} \longrightarrow
\widetilde  H_{i-1}(Y) \stackrel {\delta_{i-1}} \longrightarrow \cdots,
\end{align*}
\begin{align*}
\cdots \longrightarrow
\widetilde H_i(Y) \stackrel {\delta'_i=(\eta_i,\varphi_i)} {\longrightarrow}
\widetilde  H_i(\Delta-[r]) \bigoplus \widetilde H_i(X^{[r]}) \longrightarrow
\widetilde  H_i(\Gamma) \stackrel {\rho_i} \longrightarrow
\widetilde  H_{i-1}(Y) \stackrel {\delta'_{i-1}} \longrightarrow \cdots.
\end{align*}
The maps $\eta_i,\psi_i,\varphi_i$ are induced by inclusions.
Since $Y \subset \st(1) \subset X$
and since $\st(1)$ and $X^{[r]}$ are acyclic,
the maps $\psi_i$ and $\varphi_i$ are zero by Lemma \ref{3.1}.
Hence
$$\Image \delta_i \cong \Image \delta'_i$$
for all $i$.
Then the long exact sequences above tell
\begin{align*}
\dim_\kk \widetilde H_i(\Delta) 
& = \dim_\kk \widetilde H_{i-1}(Y)+\dim_\kk \widetilde H_i(\Delta-[r])+ \dim_\kk \widetilde H_i(X)\\
&\hspace{16pt}-\dim_\kk (\Image \delta_i) -\dim_\kk (\Image \delta_{i-1})\\
&\geq 
\dim_\kk \widetilde H_{i-1}(Y)+\dim_\kk \widetilde H_i(\Delta-[r])\\
&\hspace{16pt}-\dim_\kk (\Image \delta'_i) -\dim_\kk (\Image \delta'_{i-1})\\
&= \dim_\kk \widetilde H_i(\Gamma)
\end{align*}
for all $i$ as desired.
\end{proof}

We now complete the proof of the Lemma \ref{3.2}.
Consider the Mayer--Vietris long exact sequence for $\Sigma=\Sigma' \cup \Delta^{[r]}$
\begin{align*}
\cdots \longrightarrow
\widetilde H_i(\Sigma') \bigoplus \widetilde H_i(\Delta^{[r]}) \longrightarrow
\widetilde H_i(\Sigma)  \longrightarrow
\widetilde H_{i-1}(\Gamma) \longrightarrow \cdots.
\end{align*}
Since $\Delta^{[r]}$ is acyclic,
if $\widetilde H_i(\Sigma) \ne 0$,
then either $\widetilde H_{i-1}(\Gamma)\ne 0$ or $\widetilde H_i(\Sigma') \ne 0$.
However, the induction hypothesis tells that $\widetilde H_i(\Sigma') \ne 0$ implies $\widetilde H_{i-1}(\Delta) \ne 0$ and the claim tells that $\widetilde H_{i-1}(\Gamma) \ne 0$ implies $\widetilde H_{i-1}(\Delta) \ne 0$.
Hence $\widetilde H_i(\Sigma) \ne 0$ implies $\widetilde H_{i-1}(\Delta) \ne 0$.
\end{proof}

We are now ready to prove Proposition \ref{3.3}.
Just in case, we recall the statement.

\begin{proposition3}
Let $\II \subset S_\infty$ be a symmetric monomial ideal generated by partitions of length $\leq m$, and let $I_n =\II \cap S_n$ for $n \geq 1$.
For any integer $n \geq m$ and vector $\aaa=(a_1,\dots,a_t,b,\dots,b) \in \ZZ^n_{\geq 0}$ with $a_1 \geq \cdots \geq a_t >b \geq 1$,
we have
$$\beta_{i,\aaa}(I_n) \ne 0 \ \Leftrightarrow \beta_{i+1,(\aaa,b)}(I_{n+1}) \ne 0.$$
\end{proposition3}

\begin{proof}[Proof of Proposition \ref{3.3}]
We note $t <m$ by Proposition \ref{2.4}.
Let
$$\Delta=\Delta_\aaa^{I_n}=\{ F \subset [n] : x^{\ua} \cdot x^{[n] \setminus F} \in I_n\}$$
and
$$\Sigma=\Delta_{(\aaa,b)}^{I_{n+1}}=\{ F \subset [n+1]: x^\ua \cdot  x_{n+1}^{b-1} \cdot  x^{[n+1]\setminus F} \in I_{n+1}\}.$$
Clearly, $\Delta$ and $\Sigma$ are fixed by permutations on $\{t+1,\dots,n\}$ and on $\{t+1,\dots,n+1\}$, respectively.
We claim
\begin{align}
\label{3-3}
\Sigma=\Delta^{[t+1]} \cup \Delta^{[t+2]} \cup \cdots \cup \Delta^{[n]} \cup \Delta^{[n+1]}.
\end{align}
To prove this equality, we first show that the righthand side is fixed by permutations on $\{t+1,\dots,n+1\}$.
Let $\tau_{i,j} \in \sym_{n+1}$ be the transposition of $i$ and $j$. Since $\Delta$ is fixed by permutations on $\{t+1,\dots,n\}$, so does $\bigcup_{k=t+1}^{n+1} \Delta^{[k]}$. Thus it is enough to show that, for any $F \in \bigcup_{k=t+1}^{n+1} \Delta^{[k]}$ and $k \in \{t+1,\dots,n\}$, we have $\tau_{k,n+1} (F) \in \bigcup_{k=t+1}^{n+1} \Delta^{[k]}$.
To see this, we may assume either (i) $n+1 \in F$ and $k \not \in F$, or (ii) $n+1 \not \in F$ and $k \in F$. In the former case, $F$ must be in $\Delta^{[n+1]}$ since $\Delta$ is a simplicial complex on $[n]$, and we have $F \setminus \{n+1\} \in \Delta$, which implies $\tau_{k,n+1}(F)= (F \setminus \{n+1\}) \cup \{k\} \in \Delta^{[k]}$.
In the latter case, $F \in \Delta^{[\ell]}$ for some $ t+1 \leq \ell \leq n$.
If $\ell \in F$ then $F=\tau_{k,\ell}(F) \in \Delta^{[k]}$ and if $\ell \not \in F$ then $F \in \Delta$.
In both cases, we have $F \setminus \{k\} \in \Delta$ which implies $\tau_{k,n+1}(F)=(F \setminus \{k\}) \cup \{n+1\} \in \Delta^{[n+1]}$.

Now we prove the inclusion `$\supset$'.
Let $F \in \Delta$. Since $x^\ua \cdot x^{[n]\setminus F} \in I_n$, it is clear that
$$x^\ua \cdot x_{n+1}^{b-1} \cdot x^{[n+1] \setminus (F \cup \{n+1\})} = x^\ua \cdot x_{n+1}^{b-1} \cdot x^{[n]\setminus F} \in I_{n+1},$$
which implies $F \cup \{n+1\} \in \Sigma$.
Hence $\Delta^{[n+1]} \subset \Sigma$. This also tells $\Delta^{[k]} \subset \Sigma$ for all $k=t+1,\dots,n$ since $\Sigma$ is fixed by permutations on $\{t+1,\dots,n+1\}$.

Next we prove the inclusion `$\subset$' in \eqref{3-3}.
Let $F \in \Sigma$. Then
$$m=x^\ua \cdot x_{n+1}^{b-1} \cdot  x^{[n+1]\setminus F} \in I_{n+1}.$$
Since both simplicial complexes in \eqref{3-3} are fixed by permutations on $\{t+1,\dots,n+1\}$,
we may assume either (i) $F$ contains none of $t+1,\dots,n+1$, or (ii) $n+1 \in F$.
The desired inclusion follows from the following case analysis.

{\bf Case (i).}
If $F$ contains none of $t+1,\dots,n+1$,
then
$$m= x_1^{a_1-1} \cdots x_t^{a_t-1} x_{t+1}^b \cdots x_{n+1}^b \cdot x^{[t]\setminus F} \in I_{n+1}.$$
Then Lemma \ref{tech} tells $x^\ua \cdot x^{[n]\setminus F} = (m/x_{n+1}^b) \in I_{n+1}\cap S_n=I_n$
and we have $F \in \Delta$.

{\bf Case (ii).}
Suppose $n+1 \in F$. Let $F'=F \setminus \{n+1\}$.
We have
$$m= x_1^{a_1-1} \cdots x_t^{a_t-1} x_{t+1}^{b-1} \cdots x_{n}^{b-1} \cdot x^{[n]\setminus F'} \cdot x_{n+1}^{b-1} \in I_{n+1}.$$
Again, Lemma \ref{tech} tells $x^\ua \cdot x^{[n]\setminus F'} \in I_{n+1}$ and we have $F' \in \Delta$.
Hence we get $F=F' \cup \{n+1\} \in \Delta^{[n+1]}$.

Now we prove the proposition. Let
$$\Sigma'=\Delta^{[t+1]} \cup \cdots \cup \Delta^{[n]}.$$
Then $\Sigma=\Sigma' \cup \Delta^{[n+1]}$ and $\Sigma' \cap \Delta^{[n+1]}=\Delta$ since $n+1$ is not a vertex of $\Sigma'$.
Consider the Mayer--Vietris long exact sequence for $\Sigma=\Sigma' \cup \Delta^{[n+1]}$
\begin{align*}
\cdots
\longrightarrow
\widetilde H_i(\Delta)
\stackrel {\delta_i=(\psi_i,\varphi_i)} \longrightarrow
\widetilde H_i(\Sigma') \bigoplus \widetilde H_i(\Delta^{[n+1]})
\longrightarrow
\widetilde H_i(\Sigma)
\longrightarrow
\widetilde H_{i-1}(\Delta)
\stackrel {\delta_{i-1}} \longrightarrow
\cdots.
\end{align*}
Note that $\psi_i:\widetilde H_i(\Delta) \to \widetilde H_i(\Sigma')$ and $\varphi_i: \widetilde H_i(\Delta) \to \widetilde H_i(\Delta^{[n+1]})$
are the maps induced by inclusions.
Since $\Delta \subset \Delta^{[1]} \subset \Sigma'$, Lemma \ref{3.1} tells that $\delta_i$ is a zero map for all $i$.
Then we have $\widetilde H_i(\Sigma) \ne 0$ if and only if $\widetilde H_i(\Sigma') \ne 0$ or $\widetilde H_{i-1}(\Delta)\ne 0$.
However, Lemma \ref{3.2} tells that $\widetilde H_i(\Sigma')\ne 0$ implies $\widetilde H_{i-1}(\Delta)\ne 0$.
Hence $\widetilde H_i(\Sigma) \ne 0$ if and only if $\widetilde H_{i-1}(\Delta) \ne 0$, in particular,
$$\beta_{i,\aaa}(I_n) \ne 0
\Leftrightarrow
\widetilde H_{i-1}(\Delta) \ne 0
\Leftrightarrow
\widetilde H_i(\Sigma) \ne 0
\Leftrightarrow
\beta_{i+1,(\aaa,b)}(I_{n+1}) \ne 0,$$
as desired.
\end{proof}

\section{Questions and Problems}

In this section we pose some questions.

\subsection*{Algebraic proof}
Our proof of Theorem \ref{1.2} is based on careful analysis of homologies of simplicial complexes,
and we do not see any {\em algebraic} reason that explains why we get such a simple stability for the shape of Betti tables.
Thus we ask the following question.

\begin{question}
\label{4.1}
Is there a purely algebraic proof of Theorem \ref{1.2}?
\end{question}

\subsection*{Betti numbers}
Let $\II \subset S_\infty$ be a symmetric monomial ideal generated by partitions of length $\leq m$ and $I_n=\II \cap S_n$ for $n \geq 1$.
We determine the shape of the (multigraded) Betti table of $I_n$, however we could tell nothing about the values of Betti numbers.

\begin{question}
With the same notation as above,
is there a way to determine the numbers $\beta_{i,j}(I_n)$ for all $i,j$ and $n \geq m$?
\end{question}

To study such a problem,
it might be important to understand the $\kk[\sym_n]$-module structure of $\Tor_i(I_n,\kk)_j$ rather than just determining its $\kk$-dimension.
Such a structure is studied for some particular monomial ideals in \cite{Ga,BdA}.
Also, a more challenging problem would be

\begin{problem}
Find a way to construct a minimal ($\sym_{m+1}$-equivariant) free resolution of $I_{m+1}$ from a given minimal ($\sym_m$-equivariant) free resolution of $I_m$.
\end{problem}

\subsection*{Combinatorics of partitions}
The set $\Lambda$ of all partitions has a natural poset structure if we define its order by $\lambda \leq \mu$ if $x^\lambda$ divides $x^\mu$.
We say that a subset $P \subset \Lambda$ is an ideal (or sometimes called a filter) of $\Lambda$ 
if for any $\lambda \in P$ and $\mu > \lambda$ one has $\mu \in P$.
Then, for any symmetric monomial ideal $\II \subset S_\infty$,
the set $P(\II)=\{\lambda \in \Lambda: x^\lambda \in \II\}$ is an ideal of $\Lambda$.
It is not hard to see that the assignment $\II \to P(\II)$ gives a one to one correspondent between symmetric monomial ideals in $S_\infty$ and ideals in $\Lambda$ (see \cite{BdA}).
This natural correspondence suggests the following problem.

\begin{problem}
Study relations between algebraic properties of symmetric monomial ideals in $S_\infty$ (or $S_n$) and combinatorial properties of ideals in $\Lambda$.
\end{problem}

As we remarked in the introduction,
Raicu \cite{Ra} find a combinatorial formula of the projective dimension and the regularity of a symmetric monomial ideal $I$ in $S_n$.
It would be interesting to find a combinatorial way to determine $\beta_{i,j}(I)$ (or more strongly $\kk[\sym_n]$-module structure of $\Tor_i(I,\kk)_j$) from combinatorial information of $P(I)$.

We note that the graded Betti numbers may depend on the characteristic of $\kk$.
For example, if $I \subset S_6$ is generated by the following $10$ partitions
\begin{align*}
&(5,4,4,2,2,1),(5,4,4,3,1,1),(5,5,3,3,1,1),(5,5,3,3,2),(5,5,4,2,2),\\
&(6,4,3,2,2,1),(6,4,3,3,2),(6,4,4,3,1),(6,5,3,2,1,1),(6,5,4,2,1),
\end{align*}
then $\Delta^I_{(6,5,4,3,2,1)}$ is a triangulation of $\RR P^2$
and $\beta_{2,(6,5,4,3,2,1)}(I)$ is $0$ if $\mathrm{char}(\kk) \ne 2$ and $1$ if $\mathrm{char}(\kk)=2$.
On the other hand, the result of Raicu tells that projective dimension and regularity do not depend on the characteristic of the ground field.

\subsection*{Generalization}
The study of symmetric monomial ideals in $S_\infty$ is a very special case of a more general setting studied by Le, Nagel, Nguyen and R\"omer
\cite{LNNR1,LNNR2}, who consider homogeneous ideals in $R=\kk[x_{i,j}:1\leq i \leq c, j \geq 1]$ fixed by the action of the monoid $\mathrm{inc}(\NN)=\{\sigma :\NN \to \NN \mid \sigma(i+1) > \sigma(i) \mbox{ for all }i\}$.
It would be an attractive problem to find generalizations of the results in this paper to such a more general setting.

\begin{question}
Can we find a (possibly weaker) version of Theorem \ref{1.2} which holds for all homogeneous ideals, ideals in $R$, or ideals fixed by $\mathrm{inc}(\NN)$-action?
\end{question}

The case when $\II$ is not a monomial ideal or when $\II$ is not fixed by $\sym_\infty$-action seems to be more complicated. Below we write a few example showing that the situation is not as simple as the case of monomial ideals in $S_\infty$.

\begin{example}
\label{fex1}
It was pointed out by Hop Nguyen that the statement of Corollary \ref{1.4} (and therefore Theorem \ref{1.2} also) does not hold for $\mathrm{inc}(\NN)$-invariant monomial ideals. Indeed, if 
$$\II = (\sigma(x_1^2x_2):\sigma \in \mathrm{inc}(\NN)) + (\sigma(x_1^3): \sigma \in \mathrm{inc}(\NN))$$
and $I_n = \II \cap S_n$, then $x_1x_2 \cdots x_{k -1} x_k^2$ is a socle element of $S_n/I_n$ for $k=1,2,\dots,n$.
This tells that 
$|\{j: \beta_{n-1,j}(I_n) \ne 0\}| \geq  n$,
so even Corollary \ref{1.4} does not hold for $\mathrm{inc}(\NN)$-invariant monomial ideals.
\end{example}

\begin{example}
There are symmetric homogeneous ideals whose Betti table has similar shape as ideals in Example \ref{fex1}.
Let $n \geq k$ be positive integers
and let 
$$
\textstyle
I_{n,k}^{\mathrm{Vd}}= \left( \prod_{1 \leq p <q \leq k} (x_{i_p}-x_{i_q}): 1 \leq i_1 < \cdots < i_k \leq n \right) \subset S_n.$$
The ideal $I_{n,k}^{\mathrm{Vd}}$ is fixed by an action of $\sym_n$,
so we get a chain of symmetric ideals
 $$\cdots \subset I_{n-1,k}^{\mathrm{Vd}} \subset I_{n,k}^{\mathrm{Vd}}\subset I_{n+1,k}^{\mathrm{Vd}} \subset \cdots.$$
The minimal free resolution of $I_{n,k}^{\mathrm{Vd}}$ was given by Yanagawa and Watanabe \cite{WY}.
Their result tells
$\lim_{n \to \infty} |\{j: \beta_{n-k-1,j}(I_{n,k}^{\mathrm{Vd}}) \ne 0\}| = \infty.$
See \cite[Theorem 2.2 and Corollary 2.3]{WY}.
\end{example}

\begin{example}
Consider the symmetric ideal
$$I_n = (x_i^2: i= 1,2,\dots,n) + (x_ix_j -x_kx_l: 1 \leq i < j \leq n, 1\leq k < l \leq n) \subset S_n.$$
The ideal is the defining ideal of the graded M\"obius algebra for a rank $1$ simple matroid \cite{MN,HW}.
The ideal $I_n$ is clearly fixed by the action of $\sym_n$,
and it is not hard to see that $S_n/I_n$ is an Artinian Gorenstein algebra with Hilbert function $(1,n,1)$. So, for any integer $n \geq 2$, the Betti table of $I_n$ has the following form.
\begin{center}
\begin{tabular}{c|cccccccccc}
 & 0 & 1 & 2 & $\cdots$ & $\cdots$ & $n\! -\! 2$ & $n\!-\!1$\\
\hline
2 & $*$ & $*$ & $*$ & $\cdots$ & $\cdots$ & $*$ & \\
3 &  &  &  &  &  &  & $1$
 \end{tabular}
\end{center}
The shape of the Betti table is clearly not a union of line segments.
\end{example}

\noindent
\textbf{Acknowledgments}:
The author thanks Hop Nguyen, Claudiu Raicu and Tim R\"omer for useful comments. The research of the author is partially supported by KAKENHI 16K05102.

\end{document}